\documentclass[11pt,a4paper]{amsart}
\usepackage{amssymb, amstext, amscd, amsmath}
\usepackage{url}
\usepackage{amsfonts}
\usepackage{lscape}
\usepackage{amsthm}
\usepackage{amsfonts}
\usepackage{array}
\usepackage{eucal}
\usepackage{latexsym}
\usepackage{mathrsfs}
\usepackage{textcomp}
\usepackage{verbatim}
\usepackage{tikz}

\newtheorem{thm}{Theorem}[section]

\newtheorem{lem}[thm]{Lemma}

\numberwithin{equation}{section}

\newcommand{\bQ}{{\mathbb{Q}}}
\newcommand{\bR}{{\mathbb{R}}}

  \newcommand{\C}{{\mathcal{C}}}

  \newcommand{\M}{{\mathcal{M}}}

\renewcommand{\P}{{\mathcal{P}}}

  \newcommand{\Y}{{\mathcal{Y}}}

\renewcommand{\span}{\operatorname{span}}
\newcommand{\rank}{\operatorname{rank}}
\newcommand{\coker}{\operatorname{coker}}

\newcommand{\td}{\mathrm{td}\,}
\newcommand{\sm}{\setminus}
\newcommand{\CYL}{\mathrm{cyl}}

\newcommand{\VR}{\mathrm{VR}}

\begin{document}

\title[Global rigidity of generic frameworks on the cylinder]{Global rigidity of generic frameworks on the cylinder}
\author[Bill Jackson]{Bill Jackson}
\address{School of Mathematical Sciences\\ Queen Mary, University of London\\
E1 4NS \\ U.K. }
\email{b.jackson@qmul.ac.uk}
\author[Anthony Nixon]{Anthony Nixon}
\address{Department of Mathematics and Statistics\\ Lancaster University\\
LA1 4YF \\ U.K. }
\email{a.nixon@lancaster.ac.uk}
\date{\today}

\begin{abstract}
We show that a generic framework $(G,p)$ on the cylinder is globally
rigid if and only if $G$ is a  complete graph on at most four
vertices or $G$ is both redundantly rigid and $2$-connected. To
prove the theorem we also derive a new recursive construction of
circuits in the simple $(2,2)$-sparse matroid, and a characterisation of
rigidity for generic frameworks on the cylinder when a single  designated
vertex is allowed to move off the cylinder.
\end{abstract}


\maketitle

\section{Introduction}

A (bar-joint) framework $(G,p)$ in $\mathbb{R}^d$ is the combination of a
finite, simple graph $G=(V,E)$ and a  realisation $p:V\rightarrow
\mathbb{R}^d$. The framework $(G,p)$ is rigid if every edge-length
preserving continuous motion of the vertices arises as a congruence
of $\mathbb{R}^d$. Moreover $(G,p)$ is globally rigid if every framework $(G,q)$ with the same edge lengths as $(G,p)$ arises from a congruence of $\mathbb{R}^d$.

In general it is NP-hard to determine the rigidity or global rigidity of a given framework \cite{Abb,Sax}.  These problems become more tractable, however, for {\em generic frameworks in $\bR^d$}, i.e. frameworks whose coordinates are algebraically independent over $\bQ$. It is known that both the rigidity and global rigidity of a generic framework $(G,p)$ in $\mathbb{R}^d$ depend only on the underlying graph $G$, see \cite{AR,GHT}. We say that  {\em $G$ is rigid} or {\em globally rigid in $\mathbb{R}^d$} if some/every generic realisation of $G$ in $\mathbb{R}^d$ has the corresponding property.   Combinatorial characterisations of generic rigidity and global rigidity in $\mathbb{R}^d$ have been obtained when $d\leq 2$, see \cite{laman,J&J}, and these characterisations give rise to efficient combinatorial algorithms to decide if these properties hold. In higher dimensions, however, no  combinatorial characterisations or algorithms are yet known.

We consider the situation where $(G,p)$ is a framework in
$\mathbb{R}^3$ whose vertices are constrained to lie on a fixed surface.
Combinatorial characterisations for generic rigidity in this context were established for the sphere \cite{Wcones,NOP},  cylinder \cite{NOP}, and cone \cite{NOP2}. In particular it was shown that a generic realisation of a graph $G$ on the sphere is rigid if and only $G$ is rigid in the plane.

The characterisation of rigidity for a generic framework $(G,p)$ on the cylinder uses the {\em simple $(2,2)$-sparse matroid}. This is the matroid $\M_{2,2}^*$ on the edge set of a large complete graph $K_n$ in which a set of edges $F$ is {\em independent} if and only if $|F'|\leq 2|V(F')|-2$ for all $\emptyset\neq F'\subseteq F$, with strict inequality when $|F'|=2$. For $G\subseteq K_n$, we denote the submatroid of $\M_{2,2}^*$ induced on $E(G)$ by $\M_{2,2}^*(G)$.

\begin{thm}\label{thm:cylinderlaman}
Let $(G,p)$ be a generic framework on the cylinder.
Then $(G,p)$ is
rigid if and only if $G$
is a complete graph on at most 3 vertices or $\M_{2,2}^*(G)$ has rank $2|V(G)|-2$.
\end{thm}

This characterisation implies that the rigidity of a generic framework $(G,p)$ on a cylinder depends only on the underlying graph $G$.
We will say that a graph $G$ is {\em rigid on the cylinder} if it is generically rigid on the cylinder and that  $G$ is {\em redundantly rigid on the cylinder} if $G-e$ is rigid on the cylinder for all edges $e$ of $G$.

We next consider global rigidity on surfaces. Such problems arise
naturally as sensor network localisation problems, see \cite[Page
2]{LL} and the references therein. Connelly and Whiteley \cite{CW}
showed that a graph $G$ is generically globally rigid on the sphere
if and only if it is generically globally rigid in the plane (which
holds if and only if $G$ is 3-connected and redundantly rigid in the
plane by \cite{J&J}). In \cite{JMN}, necessary combinatorial
conditions were established for a framework on a surface to be
generically globally rigid. The conditions, redundant rigidity and
$k$-connectivity (where the integer $k$ depends on the chosen
surface), are analogous to those which characterise generic global
rigidity on the plane and the sphere. These conditions were
conjectured to also be sufficient for generic frameworks on
cylinders and cones. In this paper we verify this conjecture for the
cylinder.

\begin{thm}\label{thm:full}
Let $(G,p)$ be a generic framework on the cylinder. Then $(G,p)$ is
globally rigid on the cylinder if and only if $G$ is either a
complete graph on at most four vertices or $G$ is 2-connected and
redundantly rigid on the cylinder.
\end{thm}

The key step in proving sufficiency in Theorem \ref{thm:full} is the
following result which deals with the special case when $G$ is
2-connected and redundantly rigid with the minimum possible number
of edges. Theorem \ref{thm:cylinderlaman} implies that $|E|\geq
2|V|-1$ whenever $G=(V,E)$ is redundantly rigid, and that equality
holds only if $E$ is a circuit in  $\M_{2,2}^*$. We will abuse
terminology and say that {\em $G$ is a circuit in  $\M_{2,2}^*$}
whenever this occurs. For simplicity we will refer to such graphs as
\emph{$\M_{2,2}^*$-circuits}.

\begin{thm}\label{thm:globrigid}
Let $G$ be an $\M_{2,2}^*$-circuit and $(G,p)$ be a generic
framework on the cylinder. Then $(G,p)$ is globally rigid on the cylinder.
\end{thm}


We will need the
following three results to prove Theorem \ref{thm:globrigid}. The first is a decomposition result for
$\M_{2,2}^*$-circuits which uses the graph operations defined in Figure
\ref{fig:sums}.

\begin{thm}\cite[Lemmas 3.1, 3.2, 3.3]{Nix}\label{thm:circuit_dec}
Suppose $G_0,G_1,G_2$ are graphs with $|E(G_i)|=2|V(G_i)|-1$ for all
$0\leq i\leq 2$ and that $G_0$ is a $j$-join of $G_1$ and $G_2$ for
some $1\leq j\leq 3$. Then $G_0$ is a $\M_{2,2}^*$-circuit if and
only if both $G_1$ and $G_2$ are $\M_{2,2}^*$-circuits.
\end{thm}

\begin{center}
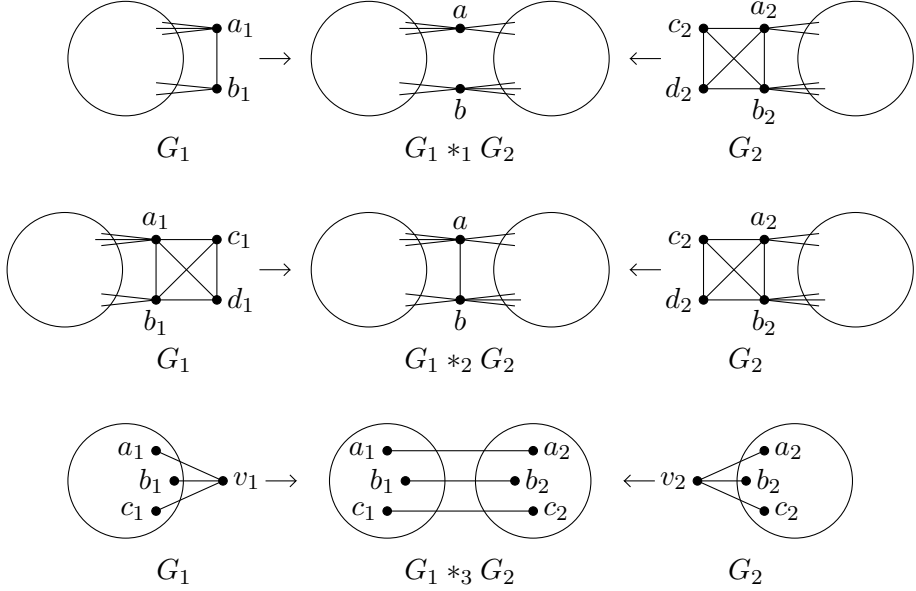
\begin{figure}[ht]
\centering
\begin{tikzpicture}[scale=0.8]

\filldraw (-4,3.5) circle (2pt) node[anchor=west]{$b_1$};
\filldraw (-4,4.5) circle (2pt) node[anchor=west]{$a_1$};

\draw (-5.5,4) circle (27pt);

\draw[black] (-5,4.5) -- (-4,4.5) -- (-4,3.5) -- (-5,3.4);

\draw[black]
(-4.9,4.6) -- (-4,4.5) -- (-4.9,4.4);

\draw[black]
(-4,3.5) -- (-5,3.6);

\filldraw (0,4.5) circle (2pt) node[anchor=south]{$a$}; 
\filldraw (0,3.5) circle (2pt) node[anchor=north]{$b$};

\draw (-1.5,4) circle (27pt);

\draw (1.5,4) circle (27pt);

\draw[black]
(-.9,4.6) -- (0,4.5) -- (-.9,4.4);

\draw[black]
(-1,4.5) -- (0,4.5);

\draw[black]
(-1,3.4) -- (0,3.5) -- (-1,3.6);

\draw[black]
(.9,4.4) -- (0,4.5) -- (.9,4.6);

\draw[black]
(.9,3.4) -- (0,3.5) -- (.9,3.6);

\draw[black]
(1,3.5) -- (0,3.5);

\draw[black] (-3.3,4) -- (-2.8,4) -- (-2.9,4.1);

\draw[black] (-2.8,4) -- (-2.9,3.9);

\filldraw (5,3.5) circle (2pt) node[anchor=north]{$b_2$};
\filldraw (5,4.5) circle (2pt) node[anchor=south]{$a_2$};
\filldraw (4,3.5) circle (2pt) node[anchor=east]{$d_2$};
\filldraw (4,4.5) circle (2pt) node[anchor=east]{$c_2$};

\draw (6.5,4) circle (27pt);


\draw[black] (5,4.5) -- (5,3.5) -- (4,3.5) -- (4,4.5) --
(5,4.5) -- (4,3.5);

\draw[black] (5,3.5) -- (4,4.5);

\draw[black] (3.3,4) -- (2.8,4) -- (2.9,3.9);

\draw[black] (2.8,4) -- (2.9,4.1);

\draw[black]
(5.9,4.4) -- (5,4.5) -- (5.9,4.6);

\draw[black]
(5.9,3.4) -- (5,3.5) -- (5.9,3.6);

\draw[black]
(6,3.5) -- (5,3.5);

  \node [rectangle, draw=white, fill=white] (b) at (-4.7,2.5) {$G_1$};

  \node [rectangle, draw=white, fill=white] (b) at (0,2.5) {$G_1 *_1 G_2$};

  \node [rectangle, draw=white, fill=white] (b) at (4.7,2.5) {$G_2$};


\filldraw (0,1) circle (2pt) node[anchor=south]{$a$}; 
\filldraw(0,0) circle (2pt) node[anchor=north]{$b$};

\filldraw (-5,0) circle (2pt) node[anchor=north]{$b_1$};
\filldraw (-5,1) circle (2pt) node[anchor=south]{$a_1$}; 
\filldraw (-4,0) circle (2pt) node[anchor=west]{$d_1$}; 
\filldraw (-4,1) circle (2pt) node[anchor=west]{$c_1$};

\filldraw (5,0) circle (2pt) node[anchor=north]{$b_2$};
 \filldraw (5,1) circle (2pt) node[anchor=south]{$a_2$}; 
\filldraw (4,0) circle (2pt) node[anchor=east]{$d_2$}; 
\filldraw (4,1) circle (2pt) node[anchor=east]{$c_2$};

\draw (-6.5,.5) circle (27pt);

\draw (6.5,.5) circle (27pt);

\draw (-1.5,.5) circle (27pt);

\draw (1.5,.5) circle (27pt);

%
%
%

\draw[black] (-5,1) -- (-5,0) -- (-4,0) -- (-4,1) --
(-5,1) -- (-4,0);

\draw[black] (-5,0) -- (-4,1);

\draw[black] (5,1) -- (5,0) -- (4,0) -- (4,1) --
(5,1) -- (4,0);

\draw[black] (5,0) -- (4,1);

\draw[black] (0,1) -- (0,0);

\draw[black] (-3.3,.5) -- (-2.8,.5) -- (-2.9,.6);

\draw[black] (-2.8,.5) -- (-2.9,.4);

\draw[black] (3.3,.5) -- (2.8,.5) -- (2.9,.4);

\draw[black] (2.8,.5) -- (2.9,.6);

\draw[black]
(-5.9,.9) -- (-5,1) -- (-5.9,1.1);

\draw[black]
(-5.9,-.1) -- (-5,0) -- (-5.9,.1);

\draw[black]
(-6,1) -- (-5,1);

\draw[black]
(-.9,.9) -- (0,1) -- (-.9,1.1);

\draw[black]
(-.9,-.1) -- (0,0) -- (-.9,.1);

\draw[black]
(-1,1) -- (0,1);

\draw[black]
(.9,.9) -- (0,1) -- (.9,1.1);

\draw[black]
(.9,-.1) -- (0,0) -- (.9,.1);

\draw[black]
(1,0) -- (0,0);

\draw[black]
(5.9,.9) -- (5,1) -- (5.9,1.1);

\draw[black]
(5.9,-.1) -- (5,0) -- (5.9,.1);

\draw[black]
(6,0) -- (5,0);

  \node [rectangle, draw=white, fill=white] (b) at (-4.7,-1) {$G_1$};

  \node [rectangle, draw=white, fill=white] (b) at (0,-1) {$G_1 *_2 G_2$};

  \node [rectangle, draw=white, fill=white] (b) at (4.7,-1) {$G_2$};


\draw (-5.5,-3) circle (27pt); 
\draw (5.5,-3) circle (27pt); 
\draw (1.2,-3) circle (27pt); 
\draw (-1.2,-3) circle (27pt);

\filldraw (-3.9,-3) circle (2pt) node[anchor=west]{$v_1$};
\filldraw (3.9,-3) circle (2pt) node[anchor=east]{$v_2$}; 
\filldraw (-4.7,-3) circle (2pt) node[anchor=east]{$b_1$}; 
\filldraw (4.7,-3) circle (2pt) node[anchor=west]{$b_2$}; 
\filldraw (-5,-3.5) circle (2pt) node[anchor=east]{$c_1$}; 
\filldraw (-5,-2.5) circle (2pt) node[anchor=east]{$a_1$}; 
\filldraw (5,-3.5) circle (2pt) node[anchor=west]{$c_2$}; 
\filldraw (5,-2.5) circle (2pt) node[anchor=west]{$a_2$}; 
\filldraw (1.2,-2.5) circle (2pt) node[anchor=west]{$a_2$}; 
\filldraw (1.2,-3.5) circle (2pt) node[anchor=west]{$c_2$}; 
\filldraw (.9,-3) circle (2pt) node[anchor=west]{$b_2$}; 
\filldraw (-1.2,-2.5) circle (2pt) node[anchor=east]{$a_1$}; 
\filldraw (-1.2,-3.5) circle (2pt) node[anchor=east]{$c_1$}; 
\filldraw (-.9,-3) circle (2pt) node[anchor=east]{$b_1$};

\draw[black] (-.9,-3) -- (.9,-3);

\draw[black] (-1.2,-2.5) -- (1.2,-2.5);

\draw[black] (-1.2,-3.5) -- (1.2,-3.5);

\draw[black] (-5,-3.5) -- (-3.9,-3) -- (-4.7,-3);

\draw[black] (-3.9,-3) -- (-5,-2.5);

\draw[black] (5,-3.5) -- (3.9,-3) -- (4.7,-3);

\draw[black] (3.9,-3) -- (5,-2.5);

\draw[black] (-3.2,-3) -- (-2.7,-3) -- (-2.8,-3.1);

\draw[black] (-2.7,-3) -- (-2.8,-2.9);

\draw[black] (3.2,-3) -- (2.7,-3) -- (2.8,-3.1);

\draw[black] (2.7,-3) -- (2.8,-2.9);

  \node [rectangle, draw=white, fill=white] (b) at (-4.7,-4.5) {$G_1$};

  \node [rectangle, draw=white, fill=white] (b) at (0,-4.5) {$G_1 *_3 G_2$};

  \node [rectangle, draw=white, fill=white] (b) at (4.7,-4.5) {$G_2$};

\end{tikzpicture}
\caption{The 1-, 2- and 3-join operations. The 1- and 2-join
operations form the graphs in the centre by merging $a_1$ and $a_2$
into $a$, and $b_1$ and $b_2$ into $b$. We write $G=G_1*_jG_2$ to mean $G$ is a $j$-join of $G_1$ and $G_2$.} \label{fig:sums}
\end{figure}
\end{center}

The second result we shall need is a recursive construction for
$\M_{2,2}^*$-circuits which uses the $j$-join operations as well as
the {\em $1$-extension operation} which deletes an edge $xy$ from a
graph $G$ and then adds a new vertex $v$  and three new edges
$vx,vy,vz$ for some vertex $z\neq x,y$. The recursion begins with
the three  $\M_{2,2}^*$-circuits defined in Figure
\ref{fig:smallgraphs}.

 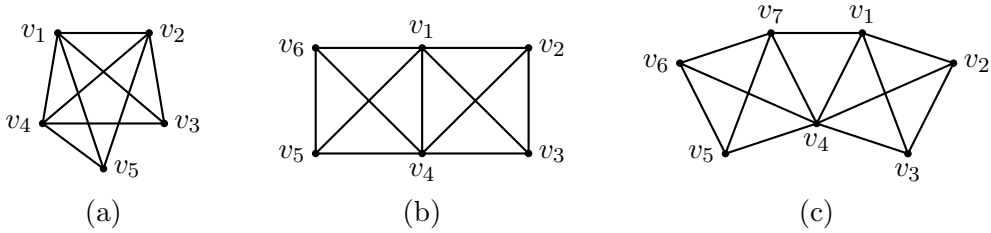
\begin{figure}[htp]
\begin{center}
\begin{tikzpicture}[scale=.4]
\filldraw (-.5,0) circle (3pt)node[anchor=east]{$v_4$};
\filldraw (0,3) circle (3pt)node[anchor=east]{$v_1$};
\filldraw (3.5,0) circle (3pt)node[anchor=west]{$v_3$};
\filldraw (3,3) circle (3pt)node[anchor=west]{$v_2$};
\filldraw (1.5,-1.5) circle (3pt)node[anchor=west]{$v_5$};

 \draw[black,thick]
(1.5,-1.5) -- (-.5,0) -- (0,3) -- (3.5,0) -- (3,3) -- (-.5,0) -- (3.5,0);

\draw[black,thick]
(1.5,-1.5) -- (0,3) -- (3,3) -- (1.5,-1.5);

  \node [rectangle, draw=white, fill=white] (b) at (1.5,-3) {(a)};
        \end{tikzpicture}
          \hspace{0.5cm}
     \begin{tikzpicture}[scale=.4]
\filldraw (0,0) circle (3pt)node[anchor=east]{$v_5$};
\filldraw (0,3.5) circle (3pt)node[anchor=east]{$v_6$};
\filldraw (3.5,0) circle (3pt)node[anchor=north]{$v_4$};
\filldraw (3.5,3.5) circle (3pt)node[anchor=south]{$v_1$};
\filldraw (7,0) circle (3pt)node[anchor=west]{$v_3$};
\filldraw (7,3.5) circle (3pt)node[anchor=west]{$v_2$};

 \draw[black,thick]
(0,0) -- (0,3.5) -- (3.5,0) -- (3.5,3.5) -- (0,0) -- (3.5,0) -- (7,3.5);

\draw[black,thick]
(0,3.5) -- (3.5,3.5) -- (7,3.5) -- (7,0) -- (3.5,3.5);

\draw[black,thick]
(7,0) -- (3.5,0);

\node [rectangle, draw=white, fill=white] (b) at (3.5,-2) {(b)};
\end{tikzpicture}
       \hspace{0.5cm}
    \begin{tikzpicture}[scale=.4]
\filldraw (0,-1) circle (3pt)node[anchor=east]{$v_5$};
\filldraw (-1.5,2) circle (3pt)node[anchor=east]{$v_6$};
\filldraw (3,0) circle (3pt)node[anchor=north]{$v_4$};
\filldraw (1.5,3) circle (3pt)node[anchor=south]{$v_7$};
\filldraw (6,-1) circle (3pt)node[anchor=north]{$v_3$};
\filldraw (4.5,3) circle (3pt)node[anchor=south]{$v_1$};
\filldraw (7.5,2) circle (3pt)node[anchor=west]{$v_2$};

 \draw[black,thick]
(0,-1) -- (-1.5,2) -- (3,0) -- (1.5,3) -- (0,-1) -- (3,0) -- (4.5,3);

\draw[black,thick]
(-1.5,2) -- (1.5,3) -- (4.5,3) -- (6,-1) -- (7.5,2) -- (4.5,3);

\draw[black,thick]
(6,-1) -- (3,0) -- (7.5,2);

  \node [rectangle, draw=white, fill=white] (b) at (3,-3) {(c)};
         \end{tikzpicture}
\end{center}
\vspace{-0.3cm}
\caption{The graphs $K_5^-,H_1$ and $H_2$.}
\label{fig:smallgraphs}
\end{figure}

\begin{thm}{\cite[Theorem 1.1]{Nix}}\label{thm:recurse}
Suppose $G$ is an $\M_{2,2}^*$-circuit. Then $G$ can be obtained from either
$K_5^-$, $H_1$ or $H_2$ by recursively applying the operations of  $1$-extension, and  $1$-, $2$- and $3$-join.
\end{thm}

The third result tells us that two equivalent generic frameworks on the cylinder have closely related equilibrium stresses (which will be defined in Section \ref{sec:rig}).

\begin{thm}\label{thm:genstressPartial}\cite[Theorem 12]{J&N}
Let $(G,p)$ be a generic framework on the cylinder and
$(\omega,\lambda)$ be an equilibrium stress for $(G,p)$. Let $(G,q)$
be equivalent to $(G,p)$. Then $(\omega,\lambda')$ is an equilibrium
stress for $(G,q)$ for some $\lambda' \in \mathbb{R}^n$.
\end{thm}

An overview of the proof of Theorems  \ref{thm:full} and
\ref{thm:globrigid} is as follows. A key step in showing that
2-connectivity and redundant rigidity are sufficient to imply
generic global rigidity for $\M_{2,2}^*$-circuits is to show inductively that every generic
realisation of an $\M_{2,2}^*$-circuit on the cylinder has a maximum
rank equilibrium stress. The most straightforward way to do this
would be to show that the operations used in Theorem
\ref{thm:recurse} preserve the property of having such a stress, but
we are unable to show that the join operations do this. Instead we
use Theorem \ref{thm:recurse} to obtain a new recursive construction
which uses operations which we can show preserve this property. We
derive this construction in Section \ref{sec:recurse}, give formal
definitions for infinitesimal rigidity and equilibrium stresses  in
Section \ref{sec:rig}, and show that each of the operations in our
recursive construction preserves the stress property in Section
\ref{sec:vsplit}. A technical detail in the proof in Section
\ref{sec:vsplit}, that the equilibrium stress is nowhere zero, is
dealt with in Section \ref{sec:vfree} and gives rise to a
characterisation of rigidity for generic frameworks on the cylinder
in which one designated vertex is allowed to move freely in
$\mathbb{R}^3$.

Given the result that every generic realisation of an
$\M_{2,2}^*$-circuit has a maximum rank equilibrium stress,  we can
apply Theorem \ref{thm:genstressPartial} to deduce that any two
equivalent generic realisations of an $\M_{2,2}^*$-circuit on the
cylinder have related maximum rank equilibrium stresses. If we could
show that the stresses were identical then we could immediately
deduce that the two realisations are congruent using  \cite[Theorem
9]{J&N}. Since we cannot show this, we instead use the above
mentioned relation between the two stresses to first deduce that the
two frameworks are `VR-equivalent' i.e. the projections of the
vertices of the frameworks onto the axis of the cylinder are linked
by a dilation. We then apply a characterization of `global
VR-rigidity' (which we subsequently derive in Section
\ref{sec:restricted}) to deduce that the two realisations are
congruent.  This gives Theorem \ref{thm:globrigid}. We then use an
`ear-decomposition' of the rigidity matroid to extend the
characterization to arbitrary graphs and obtain Theorem
\ref{thm:full}. Both characterisations are derived in Section
\ref{sec:globalfull}. We close with a short section which points out
that the property of having a maximum rank stress matrix is a
sufficient condition for generic global rigidity on the cylinder (a
result we had previously conjectured in \cite[Conjecture 1]{J&N}),
and that our characterization of global rigidity on the cylinder
also holds for a generic framework in $\mathbb{R}^3$ with the added
constraints that the distances of the vertices from a given line are
fixed.

\section{Recursive construction}
\label{sec:recurse}

We will refine the recursive construction for $\M_{2,2}^*$-circuits
given in Theorem \ref{thm:recurse}. The aim  is to make the steps in
the recursion as simple as possible since we have to show they
preserve global rigidity, or, more precisely, preserve the property
of having a maximum rank equilibrium stress. To this end we
introduce two `new' operations. Given an $\M_{2,2}^*$-circuit
$G=(V,E)$, the first operation, {\em $K_4^-$-extension}, is just a
1-join of $G$ and $H_1$. The second operation, {\em generalised
vertex split}, is defined as follows: choose $v\in V$ and a
partition $N_1,N_2$ of the neighbours of $v$; then delete $v$ from
$G$ and add two new vertices $v_1,v_2$ joined to $N_1,N_2$,
respectively; finally add two new edges $v_1v_2,v_1x$ for some $x\in
V\sm N_1$. These operations are illustrated in Figures
\ref{fig:types_a} and \ref{fig:vsplit}.

The usual vertex splitting operation, see \cite{Whi}, is the special case when $x$ is chosen to be a neighbour of $v_2$. Note also that the special case when $v_1$ has degree 3 (and $v_2=v$) is the 1-extension operation. At times it will be convenient to work directly with the 1-extension operation itself, for example in Theorem \ref{thm:nearly3con} below.

\begin{center}
\begin{figure}[ht]
\centering
\begin{tikzpicture}[scale=0.8]

\draw (-4,10) circle (27pt);
\draw (4,10) circle (27pt);

\filldraw (-2,10.5) circle (2pt) node[anchor=south]{$v_1$};
\filldraw (-2,9.5) circle (2pt) node[anchor=north]{$v_2$};

\draw[black]
(-3.5,10.6) -- (-2,10.5) -- (-2,9.5) -- (-3.5,9.7);

\draw[black]
(-3.5,10.4) -- (-2,10.5);

\draw[black]
(-3.7,9.5) -- (-2,9.5) -- (-3.6,9.3);

\draw[black]
(0,10) -- (1,10) -- (0.9,9.9);

\draw[black]
(1,10) -- (0.9,10.1);


\filldraw (6,10.5) circle (2pt) node[anchor=south]{$v_1$};
\filldraw (6,9.5) circle (2pt) node[anchor=north]{$v_2$};
\filldraw (7,10.5) circle (2pt) node[anchor=south]{$u_1$};
\filldraw (7,9.5) circle (2pt) node[anchor=north]{$u_2$};

\draw[black]
(4.5,10.6) -- (6,10.5) -- (7,10.5) -- (7.,9.5) -- (6,9.5) -- (7,10.5);

\draw[black]
(6,9.5) -- (4.5,9.7);

\draw[black]
(4.5,10.4) -- (6,10.5) -- (7,9.5);

\draw[black]
(4.3,9.5) -- (6,9.5) -- (4.4,9.3);

\end{tikzpicture}
\caption{$K_4^-$-extension.} \label{fig:types_a}
\end{figure}
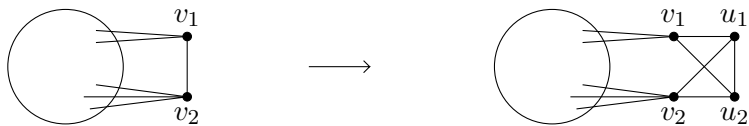

\end{center}

\begin{center}
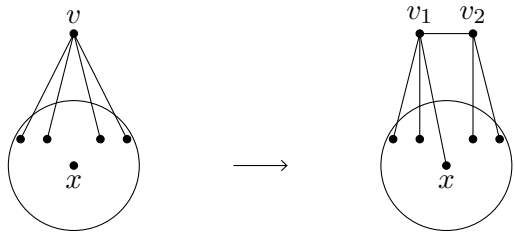
\begin{figure}[ht]
\begin{tikzpicture}[scale=0.7]
\draw (-3,-5) circle (35pt);
\draw (4,-5) circle (35pt);

\filldraw (-3,-2.5) circle (2pt) node[anchor=south]{$v$};
\filldraw (3.5,-2.5) circle (2pt) node[anchor=south]{$v_1$};
\filldraw (4.5,-2.5) circle (2pt) node[anchor=south]{$v_2$};

\filldraw (-3,-5) circle (2pt) node[anchor=north]{$x$};
\filldraw (4,-5) circle (2pt) node[anchor=north]{$x$};

\filldraw (-4,-4.5) circle (2pt);
\filldraw (-3.5,-4.5) circle (2pt);
\filldraw (-2.5,-4.5) circle (2pt);
\filldraw (-2,-4.5) circle (2pt);
\filldraw (3,-4.5) circle (2pt);
\filldraw (3.5,-4.5) circle (2pt);
\filldraw (4.5,-4.5) circle (2pt);
\filldraw (5,-4.5) circle (2pt);

\draw[black]
(0,-5) -- (1,-5) -- (0.9,-5.1);

\draw[black]
(1,-5) -- (0.9,-4.9);

\draw[black]
(-4,-4.5) -- (-3,-2.5) -- (-3.5,-4.5);

\draw[black]
(-2.5,-4.5) -- (-3,-2.5) -- (-2,-4.5);

\draw[black]
(3,-4.5) -- (3.5,-2.5) -- (3.5,-4.5);

\draw[black]
(4.5,-4.5) -- (4.5,-2.5) -- (5,-4.5);

\draw[black]
(4,-5) -- (3.5,-2.5) -- (4.5,-2.5);

\end{tikzpicture}
\caption{Generalised vertex split.} \label{fig:vsplit}
\end{figure}

\end{center}

We can now state our  simplified recursive
construction for $\M_{2,2}^*$-circuits.

\begin{thm}\label{thm:strongrecurse}
Suppose $G$ is an $\M_{2,2}^*$-circuit. Then
$G$ can be obtained from either $K_5^-$ or $H_1$ by recursively
applying the operations of  $K_4^-$-extension and generalised vertex split, in such a way that
each of the intermediate graphs is an $\M_{2,2}^*$-circuit.
\end{thm}

Since the $K_4^-$-extension
operation is a special case of the 1-join operation, it
must necessarily preserve the property of being an $\M_{2,2}^*$-circuit
by Theorem \ref{thm:circuit_dec}. The 1-extension operation also preserves the property of being an $\M_{2,2}^*$-circuit by \cite[Lemma 2.1]{Nix}, but the
generalised vertex split operation may not.

We will refer to the inverse operations to those used in Theorem
\ref{thm:strongrecurse} as {\em $K_4^-$-reduction} and {\em edge-reduction}, respectively, and to the inverse of the 1-extension operation as {\em $1$-reduction}. (Thus edge-reduction chooses two adjacent edges $e,f$ and then deletes $e$ and contracts $f$.) We say that an
application of these operations is {\em admissible} if, when
we apply it to an $\M_{2,2}^*$-circuit, we obtain a smaller $\M_{2,2}^*$-circuit. Theorem
\ref{thm:circuit_dec} implies that $K_4^-$-reduction will always be admissible but the operations of
1-reduction and, more generally, edge-reduction may not.

We say that a vertex is a \emph{node} of a graph if it has degree three, and that a node $v$ in an $\M_{2,2}^*$-circuit $G$ is {\em admissible} if we
can construct a smaller $\M_{2,2}^*$-circuit from $G$ by applying a 1-reduction
operation at $v$.
A graph is \emph{essentially $4$-edge-connected} if it is 3-edge-connected and the only edge cuts of size 3 consist of 3 edges incident to a node.
It was shown in \cite[Theorem 1.2]{Nix} that every
3-connected, essentially 4-edge-connected $\M_{2,2}^*$-circuit other than $K_5^-$
has at least two admissible nodes. We first extend this result by
allowing certain 2-vertex-cuts and 3-edge-cuts. We need the following definitions.

A {\em $2$-vertex-separation} of a graph $G$ is a pair of induced
subgraphs $F_1=(V_1,E_1)$ and $F_2=(V_2,E_2)$ such that $F_1\cup
F_2=G$, $|V_1\cap V_2|=2$ and $V_1\sm V_2\neq \emptyset\neq V_2\sm
V_1$. The 2-separation $(F_1,F_2)$ is {\em nontrivial} if $F_i\neq
K_4$ for each $i=1,2$. A {\em $3$-edge-separation} is a pair of
vertex-disjoint induced subgraphs $(F_1,F_2)$ such that $F_1\cup
F_2=G-S$ for some set $S\subseteq E$ with $|S|=3$. It is {\em
nontrivial} if $S$ is a set of three independent edges in $G$. An
{\em atom} of $G$ is a subgraph $F$ such that $F$ is an element of a
nontrivial 2-vertex separation or 3-edge-separation of $G$ and no
proper subgraph of $F$ has this property.

We also need to extend the concepts of circuits and admissible nodes to
multigraphs. The {\em $(2,2)$-sparse matroid $\M_{2,2}$} is the matroid on the edge set of a large complete multigraph in which a set of edges $F$ is independent if $|F'|\leq 2|V(F')|-2$ for all $\emptyset\neq F'\subseteq F$. Suppose that $H$ is an
{\em $\M_{2,2}$-circuit} i.e. $E(H)$ is a circuit in $\M_{2,2}$. We say that a node
$v$ of $H$ is {\em allowable} if applying the 1-reduction operation
at $v$ produces a smaller $\M_{2,2}$-circuit and does not create
any new multiple edges.

\begin{thm}\label{thm:nearly3con}
Suppose $G$ is an $\M_{2,2}^*$-circuit which is distinct from $K_5^-$, $H_1$ and $H_2$, and that $G$
has no nontrivial $2$-vertex separation and no nontrivial
$3$-edge-separation. Then $G$ has at least two admissible nodes.
\end{thm}

\begin{proof}
If $G$ is 3-connected then the statement is \cite[Theorem 1.2]{Nix}
so we may assume that $G$ is not 3-connected. Since $G$ is an $\M_{2,2}^*$-circuit,  $G$ is 2-connected by \cite[Lemma
2.3]{Nix}.
Since
$G$ has no nontrivial  2-vertex separation and $G\neq H_1$,  every
2-vertex-separation $(F_i,F_j)$ of $G$ has $F_i=K_4$ and $F_j\neq
K_4$. For every such 2-vertex-separation with $V(F_i)\cap V(F_j) =
\{x_i,y_i\}$, we consider the multigraph $H$ formed by deleting
$V(F_i)-\{x_i,y_i\}$ and $E(F_i)$ from $G$, and adding two copies of
the edge $x_iy_i$. Since $G\neq H_2$, $H$ is a 3-connected circuit
in $\M_{2,2}$ with no nontrivial 3-edge-separation. Also, each $x_i$
and $y_i$ has degree at least 4 in $H$; otherwise there would be a
nontrivial 2-vertex-separation in $G$. It follows that every
allowable node in $H$ is  an admissible node of $G$. We can now show
that $H$ contains two allowable nodes using the proof technique of
\cite[Theorem 1.2]{Nix}, see \cite[Section 4]{Nix}.
\end{proof}

Theorem \ref{thm:strongrecurse} will follow immediately from the
following reduction result.

\begin{thm}\label{strongadmiss}
Let $G=(V,E)$ be an $\M_{2,2}^*$-circuit distinct from $K_5^-$ and $H_1$. Then
$G$ has either a $K_4^-$-reduction or an admissible edge-reduction.
\end{thm}

\begin{proof}
We proceed by induction on $|V|$.  If $G$ has no nontrivial
2-separation or 3-edge-separation then either $G=H_2$ and can be
reduced to $H_1$ by an admissible edge-reduction operation, or else $G$ has two
admissible nodes by Theorem \ref{thm:nearly3con}, and hence has two admissible edge-reductions. Thus we may
suppose that this is not the case. It follows that $G$ has at least
two distinct atoms. Let $F_1$ be an atom of $G$ and $(F_1,F_2)$ be
the nontrivial separation which contains $F_1$. Consider the
following
cases.\\[1mm]
{\bf Case 1}: $(F_1,F_2)$ is a nontrivial 3-edge-separation. Let
$G_i$ be obtained from $G$ by contracting $F_{3-i}$ to a single
vertex $z_i$ for each $i\in \{1,2\}$. Then $G=G_1*_3 G_2$ so
$G_1,G_2$ are $\M_{2,2}^*$-circuits by Theorem \ref{thm:circuit_dec}. Since $F_1$
is an atom, $G_1$ has no nontrivial 2-vertex-separation or
3-edge-separation. If $G_1\not\in \{K_5^-,H_1,H_2\}$ then $G_1$ has
an admissible node $v$ distinct from $z_1$ by Theorem
\ref{thm:nearly3con}, and $v$ will be an admissible node in $G$ by Theorem \ref{thm:circuit_dec}.

If
$G_1=K_5^-$ then $G$ is as shown in Figure \ref{fig:case1a}. We can apply an edge-reduction operation which deletes the edge $xw$ and then contracts the edge $xa$ to construct the graph $G_2'$ on the right of Figure \ref{fig:case1a}. Then $G_2'$ is an $\M_{2,2}^*$-circuit since it can be obtained from $G_2$ by two $1$-extensions. Hence this edge-reduction operation is admissible. If $G_1=H_2$ then we will contradict the assumption that
$F_1$ is an atom of $G$.
 It remains to consider the subcase when
$G_1=H_1$, which is illustrated in Figure \ref{fig:case1}. In this
case the vertex $v$ is admissible, since performing a 1-reduction on
$v$ gives the graph $G_3$ on the right of Figure \ref{fig:case1},
and we have $G_3=G_2*_1 H_2$ so $G_3$ is an $\M_{2,2}^*$-circuit by Theorem
\ref{thm:circuit_dec}.

\begin{center}
\begin{figure}
\begin{tikzpicture}[scale=0.7]
\draw (0,0) circle (35pt);
\draw (10,0) circle (35pt);

\filldraw (0.3,.8) circle (2pt) node[anchor=east]{$a$};
\filldraw (0.3,0) circle (2pt) node[anchor=east]{$b$};
\filldraw (0.3,-.8) circle (2pt) node[anchor=east]{$c$};

\filldraw (3.5,0) circle (2pt) node[anchor=north]{$w$};
\filldraw (5,-.5) circle (2pt) node[anchor=north]{$z$};
\filldraw (3.5,1) circle (2pt) node[anchor=south]{$x$};
\filldraw (5,1) circle (2pt) node[anchor=south]{$y$};

\draw[black]
(0.3,.8) -- (3.5,1) -- (5,1) -- (5,-.5) -- (3.5,0) -- (0.3,0);

\draw[black]
(3.5,0) -- (3.5,1);

\draw[black]
(3.5,1) -- (5,-.5);

\draw[black]
(5,1) -- (3.5,0);

\draw[black]
(0.3,-.8) -- (5,-.5);

  \node [rectangle, draw=white, fill=white] (b) at (2.5,-2) {$G$};

\filldraw (10.3,.8) circle (2pt) node[anchor=east]{$a$};
\filldraw (10.3,0) circle (2pt) node[anchor=east]{$b$};
\filldraw (10.3,-.8) circle (2pt) node[anchor=east]{$c$};

\filldraw (13.5,0) circle (2pt) node[anchor=north]{$w$};
\filldraw (15,-.5) circle (2pt) node[anchor=north]{$z$};
\filldraw (15,1) circle (2pt) node[anchor=south]{$y$};

\draw[black]
 (10.3,.8) -- (15,1) -- (15,-.5) -- (13.5,0) -- (10.3,0);

\draw[black]
(15,1) -- (13.5,0) -- (10.3,.8);

\draw[black]
(10.3,-.8) -- (15,-.5);

\draw[black]
(6.5,0) -- (7.5,0) -- (7.4,0.1);

\draw[black]
(7.5,0) -- (7.4,-.1);

  \node [rectangle, draw=white, fill=white] (b) at (12.5,-2) {$G_2'$};

\end{tikzpicture}
\caption{The subcase $G_1=K_5^-$ in Case 1.} \label{fig:case1a}
\end{figure}
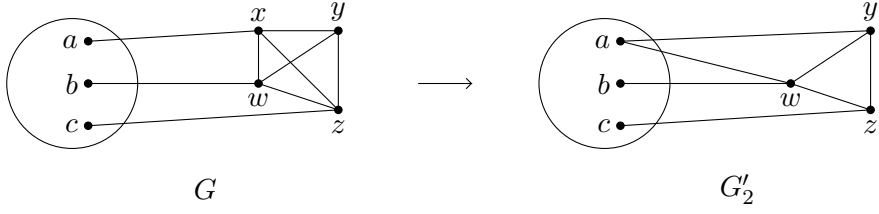
\end{center}

\begin{figure}
\begin{tikzpicture}[scale=0.7]
\draw (0,0) circle (35pt);
\draw (10,0) circle (35pt);

\filldraw (0.3,.8) circle (2pt) node[anchor=east]{$a$};
\filldraw (0.3,0) circle (2pt) node[anchor=east]{$b$};
\filldraw (0.3,-.8) circle (2pt) node[anchor=east]{$c$};

\filldraw (3.5,-.5) circle (2pt) node[anchor=north]{$w$};
\filldraw (5,-.5) circle (2pt) node[anchor=north]{$z$};
\filldraw (3.5,1) circle (2pt) node[anchor=south]{$x$};
\filldraw (5,1) circle (2pt) node[anchor=south]{$y$};

\filldraw (2.5,-1) circle (2pt) node[anchor=north]{$v$};

\draw[black]
(0.3,.8) -- (3.5,1) -- (5,1) -- (5,-.5) -- (3.5,-.5) -- (2.5,-1) -- (0.3,-.8);

\draw[black]
(0.3,0) -- (3.5,-.5) -- (3.5,1) -- (2.5,-1);

\draw[black]
(3.5,1) -- (5,-.5);

\draw[black]
(5,1) -- (3.5,-.5);

  \node [rectangle, draw=white, fill=white] (b) at (2.5,-2.5) {$G$};

\filldraw (10.3,.8) circle (2pt) node[anchor=east]{$a$};
\filldraw (10.3,0) circle (2pt) node[anchor=east]{$b$};
\filldraw (10.3,-.8) circle (2pt) node[anchor=east]{$c$};

\filldraw (13.5,-.5) circle (2pt) node[anchor=north]{$w$};
\filldraw (15,-.5) circle (2pt) node[anchor=north]{$z$};
\filldraw (13.5,1) circle (2pt) node[anchor=south]{$x$};
\filldraw (15,1) circle (2pt) node[anchor=south]{$y$};

\draw[black]
(10.3,.8) -- (13.5,1) -- (15,1) -- (15,-.5) -- (13.5,-.5);

\draw[black]
(10.3,0) -- (13.5,-.5) -- (13.5,1) -- (15,-.5);

\draw[black]
(10.3,-.8) -- (13.5,-.5);

\draw[black]
(15,1) -- (13.5,-.5);

\draw[black]
(6.5,0) -- (7.5,0) -- (7.4,0.1);

\draw[black]
(7.5,0) -- (7.4,-.1);

  \node [rectangle, draw=white, fill=white] (b) at (12.5,-2.5) {$G_3$};

\end{tikzpicture}
\caption{The subcase $G_1=H_1$ of Case 1.} \label{fig:case1}
\end{figure}
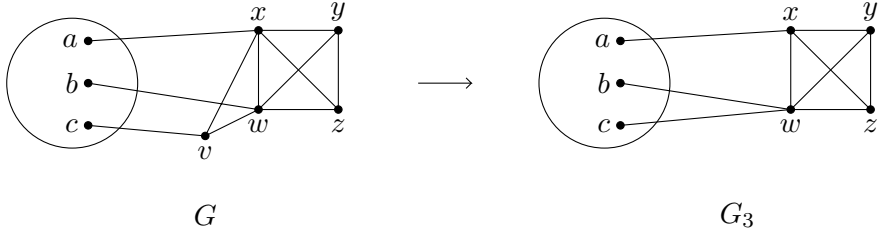

\smallskip
\noindent {\bf Case 2}: $(F_1,F_2)$ is a nontrivial
2-vertex-separation, $V(F_1)\cap V(F_2)=\{x,y\}$ and $E(F_1)\cap
E(F_2)=\{xy\}$.
Let $G_i$ be obtained from $F_i$ by adding two new vertices
$\{w,z\}$ and five new edges $\{wx,wy,wz,xz,yz\}$. Then $G=G_1*_2
G_2$ so $G_1$ and $G_2$ are $\M_{2,2}^*$-circuits by Theorem
\ref{thm:circuit_dec}. Since $F_1$ is an atom, $G_1$ has no
nontrivial 2-vertex-separation or 3-edge-separation. We have
$G_1\neq K_5^-$ since $G_1$ is not 3-connected and   $G_1 \neq H_1$
since the 2-separation $(F_1,F_2)$ is nontrivial.

Suppose $G_1 = H_2$. Then $G$ is as shown in Figure \ref{fig:case2}.
Let $G_3$ be obtained from $G$ by the edge-reduction which deletes
$xy$ and then contracts $uy$. This an admissible edge-reduction
since $G=H_2*_2 G_3$ so $G_3$ is an $\M_{2,2}^*$-circuit by Theorem
\ref{thm:circuit_dec}.

Hence we may suppose that $G_1\not\in
\{K_5^-,H_1,H_2\}$. Then $G_1$ has an admissible node $v$ by Theorem
\ref{thm:nearly3con}. The vertex $v$ will be distinct from $w,z$
since they are not admissible, and distinct from $x,y$ since they
are not nodes in $G_1$. Thus $v$ will be an admissible node in $G$.

\begin{center}
\begin{figure}
\begin{tikzpicture}[scale=0.5]
\filldraw (0,2.5) circle (3pt)node[anchor=south]{$x$};
\filldraw (0,.5) circle (3pt)node[anchor=north]{$y$};
\filldraw (-3,3) circle (3pt);
\filldraw (-4,0) circle (3pt);
\filldraw (-2,0) circle (3pt)node[anchor=north]{$u$};

\draw[black,thick]
(0,2.5) -- (0,.5);

 \draw[black,thick]
(0,.5) -- (-2,0) -- (-4,0) -- (-3,3) -- (0,2.5) -- (-4,0);

\draw[black,thick] (-3,3) -- (-2,0) -- (0,2.5);

\draw (1,1.5) circle (57pt);


%
%
%


  \node [rectangle, draw=white, fill=white] (b) at (-1,-2) {$G$};

\filldraw (1.7,1) circle (0pt)node[anchor=south]{$F_2$};

\draw[black]
(4.5,1.5) -- (6,1.5) -- (5.8,1.3);

\draw[black]
(5.8,1.7) -- (6,1.5);

\filldraw (10,2.5) circle (3pt)node[anchor=south]{$x$};
\filldraw (10,.5) circle (3pt)node[anchor=north]{$y$};
\filldraw (7,3) circle (3pt);
\filldraw (8,0) circle (3pt);

%
%

\draw (11,1.5) circle (57pt);

\draw[black,thick]
(10,2.5) -- (10,.5) -- (7,3);

 \draw[black,thick]
(10,.5) -- (8,0);

\draw[black,thick]
(7,3) -- (10,2.5);

\draw[black,thick] (7,3) -- (8,0) -- (10,2.5);


  \node [rectangle, draw=white, fill=white] (b) at (9,-2) {$G_3$};

\filldraw (11.7,1) circle (0pt)node[anchor=south]{$F_2$};
\end{tikzpicture}
\caption{The graphs $G$ and $G_3$ in the subcase $G_1=H_2$ of Case 2.} \label{fig:case2}
\end{figure}
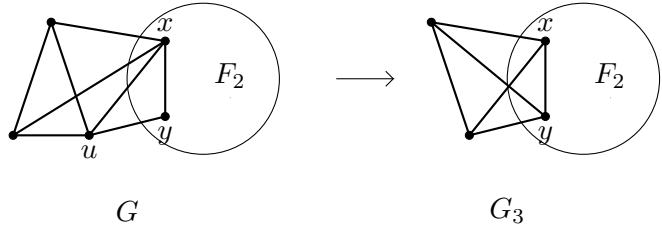
\end{center}

\noindent {\bf Case 3}: $(F_1,F_2)$ is a nontrivial
2-vertex-separation, $V(F_1)\cap V(F_2)=\{x,y\}$ and $xy\not\in E$.
We have $$|E(F_1)|+|E(F_2)|=|E(G)|=2|V(G)|-1=2|V(F_1)|+2|V(F_2)|-5$$
and $|E(F_i)|\leq 2|V(F_i)|-2$ for each $i=1,2$ so $2|V(F_i)|-3\leq
|E(F_i)|\leq 2|V(F_i)|-2$. Consider the following subcases.

\smallskip \noindent {\bf Subcase 3.1}: $|E(F_1)|= 2|V(F_1)|-3$.
Let $G_1$ be obtained from $F_1$ by adding two new vertices
$\{w,z\}$ and six new edges $\{wx,wy,wz,xy,xz,yz\}$, and
$G_2=F_2+xy$. Then $G=G_1*_1 G_2$ so $G_1$ and $G_2$ are $\M_{2,2}^*$-circuits by
Theorem \ref{thm:circuit_dec}. Since $F_1$ is an atom, $G_1$ has no
nontrivial 2-vertex-separation or 3-edge-separation. We have
$G_1\neq K_5^-$ since $G_1$ is not 3-connected. If $G_1=H_1$ then
$G_2$ is a $K_4^-$-reduction of $G$. Hence we may assume
that $G_1\neq H_1$.

Suppose $G_1=H_2$. Then $G$ is as shown in Figure \ref{fig:case3.1}.
Let $G_2'$ be obtained from $G$ by the edge-reduction which deletes
$xu$ and then contracts $uy$. This an admissible edge-reduction
since $G_2'=H_1*_1 G_2$ so $G_2'$ is an $\M_{2,2}^*$-circuit by Theorem
\ref{thm:circuit_dec}.

 Hence we may suppose that $G_1\not\in
\{K_5^-,H_1,H_2\}$. Then $G_1$ has an admissible node $v$ by Theorem
\ref{thm:nearly3con}. The vertex $v$ will be distinct from $w,z$
since they are not admissible, and distinct from $x,y$ since they
are not nodes in $G_1$. Thus $v$ will be an admissible node in $G$.

\begin{center}
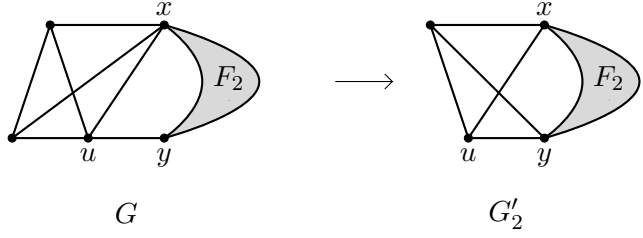
\begin{figure}
\begin{tikzpicture}[scale=0.5]

\draw[draw=gray!30!white,fill=gray!30!white] 
    plot[smooth, tension=1] coordinates{(0,0) (2.5,1.5) (0,3)}  -- 
    plot[smooth, tension=1] coordinates{(0,0) (2.5,1.5) (0,3)};

\draw[draw=white,fill=white] 
    plot[smooth, tension=1] coordinates{(0,0) (1,1.5) (0,3)};

 \draw[white,thick]
(0,0) -- (0,3);

\filldraw (0,3) circle (3pt)node[anchor=south]{$x$};
\filldraw (0,0) circle (3pt)node[anchor=north]{$y$};
\filldraw (-3,3) circle (3pt);
\filldraw (-4,0) circle (3pt);
\filldraw (-2,0) circle (3pt)node[anchor=north]{$u$};

 \draw[black,thick]
(0,0) -- (-2,0) -- (-4,0) -- (-3,3) -- (0,3) -- (-4,0);

\draw[black,thick] (-3,3) -- (-2,0) -- (0,3);

\draw[thick] plot[smooth, tension=1] coordinates{(0,0) (1,1.5)
(0,3)}; \draw[thick] plot[smooth, tension=1] coordinates{(0,0)
(2.5,1.5) (0,3)};

  \node [rectangle, draw=white, fill=white] (b) at (-1,-2) {$G$};

\filldraw (1.7,1) circle (0pt)node[anchor=south]{$F_2$};

\draw[black]
(4.5,1.5) -- (6,1.5) -- (5.8,1.3);

\draw[black]
(5.8,1.7) -- (6,1.5);

\draw[draw=gray!30!white,fill=gray!30!white] 
    plot[smooth, tension=1] coordinates{(10,0) (12.5,1.5) (10,3)}  -- 
    plot[smooth, tension=1] coordinates{(10,0) (12.5,1.5) (10,3)};

\draw[draw=white,fill=white] 
    plot[smooth, tension=1] coordinates{(10,0) (11,1.5) (10,3)};

 \draw[white,thick]
(10,0) -- (10,3);

\filldraw (10,3) circle (3pt)node[anchor=south]{$x$};
\filldraw (10,0) circle (3pt)node[anchor=north]{$y$};
\filldraw (7,3) circle (3pt);
\filldraw (8,0) circle (3pt)node[anchor=north]{$u$};

\draw[black,thick]
(10,0) -- (7,3);

 \draw[black,thick]
(10,0) -- (8,0);

\draw[black,thick]
(7,3) -- (10,3);

\draw[black,thick] (7,3) -- (8,0) -- (10,3);

\draw[thick] plot[smooth, tension=1] coordinates{(10,0) (11,1.5)
(10,3)}; \draw[thick] plot[smooth, tension=1] coordinates{(10,0)
(12.5,1.5) (10,3)};

  \node [rectangle, draw=white, fill=white] (b) at (9,-2) {$G_2'$};

\filldraw (11.7,1) circle (0pt)node[anchor=south]{$F_2$};
\end{tikzpicture}
\caption{The graphs $G$ and $G_2'$ in the subcase $G_1=H_2$ of Subcase 3.1.} \label{fig:case3.1}
\end{figure}
\end{center}

\noindent {\bf Subcase 3.2}: $|E(F_1)|= 2|V(F_1)|-2$.
Let $G_2$ be obtained from $F_2$ by adding two new vertices
$\{w,z\}$ and six new edges $\{wx,wy,wz,xy,xz,yz\}$, and
$G_1=F_1+xy$. Then $G=G_1*_1G_2$ so $G_1$ and $G_2$ are $\M_{2,2}^*$-circuits by
Theorem \ref{thm:circuit_dec}.
We may apply the argument used in Subcase 3.1 (to show that
$G_1\not\in \{K_5^-,H_1,H_2\}$) to deduce that $G_2\not \in
\{K_5^-,H_1,H_2\}$. By induction, $G_2$ has either a
 $K_4^-$-reduction or an admissible
edge-reduction.
If $G_2$ has a $K_4^-$-reduction, then
the same reduction will exist in $G$. Hence we may suppose
that $G_2$ has an admissible edge-reduction which deletes an edge
$e$ then contracts an adjacent edge $f$. If $e,f\in E(F_2)$, then
the same edge-reduction will be admissible in $G$, so we may assume that this is not the case. Since the
edge-reduction is admissible in $G_2$ we must have $f\not\in
\{wx,wy,wz,xy,xz,yz\}$, $e=xy$, and $f$ is incident to either $x$ or
$y$.
Without loss of generality we may suppose that $f$ is incident with
$y$. Let $G_2'$ be the $\M_{2,2}^*$-circuit created by this edge-reduction and
let $G_2''$ be the graph obtained from $G_2'$ by deleting the
vertices $w,z$ and adding the edge $xy$. Since $G_2'=G_2''*_1H_1$,
$G_2''$ is an $\M_{2,2}^*$-circuit by Theorem \ref{thm:circuit_dec}.

We next consider $G_1$. Since $F_1$ is an atom, $G_1$ has no
nontrivial 2-vertex-separation or 3-edge-separation.

Suppose $G_1=K_5^-$. Then $G$ has the structure of one of the
graphs shown in Figure \ref{fig:K_5^-}.
 If case (i) occurs then the 1-reduction of $G$ which deletes $v$ and adds the edge $xy$ gives the $\M_{2,2}^*$-circuit $G_2$,
 so is admissible.
 If case (ii) occurs then the edge-reduction of $G$ which deletes $h$ and contracts $g$  also gives
 the $\M_{2,2}^*$-circuit $G_2$, so is admissible.

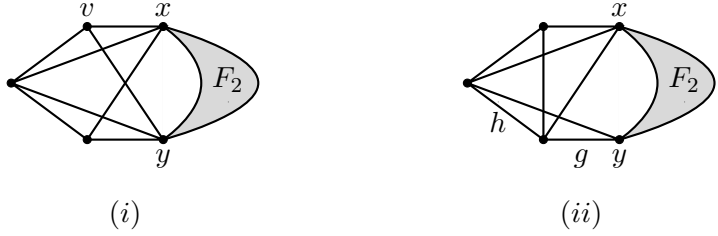
\begin{figure}
\begin{tikzpicture}[scale=0.5]

\draw[draw=gray!30!white,fill=gray!30!white] 
    plot[smooth, tension=1] coordinates{(0,0) (2.5,1.5) (0,3)}  -- 
    plot[smooth, tension=1] coordinates{(0,0) (2.5,1.5) (0,3)};

\draw[draw=white,fill=white] 
    plot[smooth, tension=1] coordinates{(0,0) (1,1.5) (0,3)};

 \draw[white,thick]
(0,0) -- (0,3);

\filldraw (0,3) circle (3pt)node[anchor=south]{$x$};
\filldraw (0,0) circle (3pt)node[anchor=north]{$y$};
\filldraw (-2,3) circle (3pt)node[anchor=south]{$v$};
\filldraw (-4,1.5) circle (3pt);
\filldraw (-2,0) circle (3pt);

 \draw[black,thick]
(0,0) -- (-2,0) -- (-4,1.5) -- (-2,3) -- (0,3) -- (-4,1.5) -- (0,0);

\draw[black,thick]
(-2,3) -- (0,0);

\draw[black,thick]
(-2,0) -- (0,3);

\draw[thick] plot[smooth, tension=1] coordinates{(0,0) (1,1.5)
(0,3)}; \draw[thick] plot[smooth, tension=1] coordinates{(0,0)
(2.5,1.5) (0,3)};

  \node [rectangle, draw=white, fill=white] (b) at (-1,-2) {$(i)$};

\filldraw (1.7,1) circle (0pt)node[anchor=south]{$F_2$};


\draw[draw=gray!30!white,fill=gray!30!white] 
    plot[smooth, tension=1] coordinates{(12,0) (14.5,1.5) (12,3)}  -- 
    plot[smooth, tension=1] coordinates{(12,0) (14.5,1.5) (12,3)};

\draw[draw=white,fill=white] 
    plot[smooth, tension=1] coordinates{(12,0) (13,1.5) (12,3)};

 \draw[white,thick]
(12,0) -- (12,3);

\filldraw (12,3) circle (3pt)node[anchor=south]{$x$};
\filldraw (12,0) circle (3pt)node[anchor=north]{$y$};
\filldraw (10,3) circle (3pt);
\filldraw (8,1.5) circle (3pt);
\filldraw (10,0) circle (3pt);

\filldraw (11,0) circle (0pt)node[anchor=north]{$g$};
\filldraw (8.8,1) circle (0pt)node[anchor=north]{$h$};

 \draw[black,thick]
(12,0) -- (10,0) -- (8,1.5) -- (10,3) -- (12,3) -- (8,1.5) -- (12,0);

\draw[black,thick]
(10,3) -- (10,0) -- (12,3);

\draw[thick] plot[smooth, tension=1] coordinates{(12,0) (13,1.5)
(12,3)}; \draw[thick] plot[smooth, tension=1] coordinates{(12,0)
(14.5,1.5) (12,3)};

  \node [rectangle, draw=white, fill=white] (b) at (11,-2) {$(ii)$};

\filldraw (13.7,1) circle (0pt)node[anchor=south]{$F_2$};
\end{tikzpicture}
\caption{The graph $G$ in the subcase $G_1=K_5^-$ of Subcase 3.2.} \label{fig:K_5^-}
\end{figure}

Suppose $G_1=H_1$. Since $F_1$ is an atom,  $G$ would have the structure of one of the graphs shown in
 Figure \ref{fig:H_1new}.  If case (i) occurs then the 1-reduction of $G$ which deletes $v$ and adds the edge $xy$ gives
 the $\M_{2,2}^*$-circuit $G_2*_2 H_2$,
 so is admissible.
 If case (ii) occurs then the edge-reduction of $G$ which deletes $h$ and contracts $g$  also gives
 the $\M_{2,2}^*$-circuit $G_2*_2 H_2$, so is admissible.

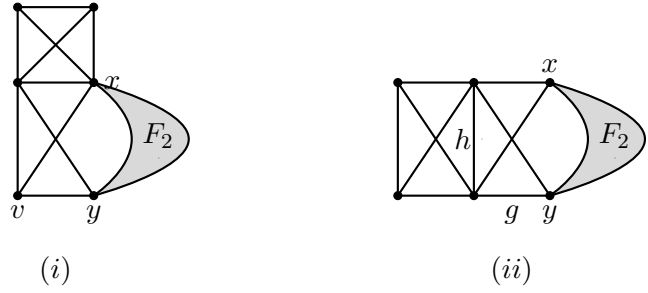
\begin{figure}
\begin{tikzpicture}[scale=0.5]

\draw[draw=gray!30!white,fill=gray!30!white] 
    plot[smooth, tension=1] coordinates{(0,0) (2.5,1.5) (0,3)}  -- 
    plot[smooth, tension=1] coordinates{(0,0) (2.5,1.5) (0,3)};

\draw[draw=white,fill=white] 
    plot[smooth, tension=1] coordinates{(0,0) (1,1.5) (0,3)};

 \draw[white,thick]
(0,0) -- (0,3);

\filldraw (0,3) circle (3pt)node[anchor=west]{$x$};
\filldraw (0,0) circle (3pt)node[anchor=north]{$y$};
\filldraw (-2,3) circle (3pt);
\filldraw (-2,0) circle (3pt)node[anchor=north]{$v$};
\filldraw (-2,5) circle (3pt);
\filldraw (0,5) circle (3pt);

 \draw[black,thick]
(0,0) -- (-2,0) --  (-2,3) -- (0,3);

\draw[black,thick]
(-2,3) -- (0,0);

\draw[black,thick]
(-2,0) -- (0,3) -- (0,5) -- (-2,5) -- (0,3);

\draw[black,thick]
(0,5) -- (-2,3) -- (-2,5);

\draw[thick] plot[smooth, tension=1] coordinates{(0,0) (1,1.5)
(0,3)}; \draw[thick] plot[smooth, tension=1] coordinates{(0,0)
(2.5,1.5) (0,3)};

  \node [rectangle, draw=white, fill=white] (b) at (-1,-2) {$(i)$};

\filldraw (1.7,1) circle (0pt)node[anchor=south]{$F_2$};



\draw[draw=gray!30!white,fill=gray!30!white] 
    plot[smooth, tension=1] coordinates{(12,0) (14.5,1.5) (12,3)}  -- 
    plot[smooth, tension=1] coordinates{(12,0) (14.5,1.5) (12,3)};

\draw[draw=white,fill=white] 
    plot[smooth, tension=1] coordinates{(12,0) (13,1.5) (12,3)};

 \draw[white,thick]
(12,0) -- (12,3);

\filldraw (12,3) circle (3pt)node[anchor=south]{$x$};
\filldraw (12,0) circle (3pt)node[anchor=north]{$y$};
\filldraw (10,3) circle (3pt);
\filldraw (10,0) circle (3pt);
\filldraw (8,0) circle (3pt);
\filldraw (8,3) circle (3pt);

\filldraw (11,0) circle (0pt)node[anchor=north]{$g$};
\filldraw (10.2,1.5) circle (0pt)node[anchor=east]{$h$};

 \draw[black,thick]
(8,3) -- (10,0) -- (8,0) -- (8,3) -- (10,3) -- (12,0) -- (10,0)  -- (10,3) -- (12,3);

\draw[black,thick]
(8,0) -- (10,3);

\draw[black,thick]
(10,0) -- (12,3);

\draw[thick] plot[smooth, tension=1] coordinates{(12,0) (13,1.5)
(12,3)}; \draw[thick] plot[smooth, tension=1] coordinates{(12,0)
(14.5,1.5) (12,3)};

  \node [rectangle, draw=white, fill=white] (b) at (11,-2) {$(ii)$};

\filldraw (13.7,1) circle (0pt)node[anchor=south]{$F_2$};
\end{tikzpicture}
\caption{The graph $G$ in the subcase $G_1=H_1$ of Subcase 3.2.} \label{fig:H_1new}
\end{figure}

Suppose $G_1=H_2$. Since $F_1$ is an atom,  $G$ would have the
structure of the graph shown in Figure \ref{fig:H_2new}(i).  We can now
use the admissible edge-reduction of $G_2$ to $G_2'$ described in the first paragraph of this subcase. Consider the edge-reduction of $G$ which deletes the edge $g$
shown in Figure \ref{fig:H_2new}(i) and contracts the edge $f\in
E(F_2)$. This is admissible since it gives the $\M_{2,2}^*$-circuit $G_2''*_1 H_3$,
where $H_3$ is the $\M_{2,2}^*$-circuit shown in Figure \ref{fig:H_2new}(ii).
(Note that $H_3$ is an $\M_{2,2}^*$-circuit since $H_3=H_1*_1H_2$.)

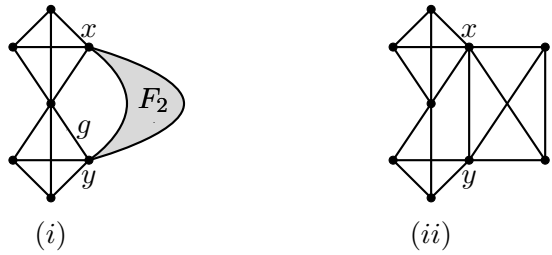
\begin{figure}
\begin{tikzpicture}[scale=0.5]
\draw[draw=gray!30!white,fill=gray!30!white] 
    plot[smooth, tension=1] coordinates{(0,0) (2.5,1.5) (0,3)}  -- 
    plot[smooth, tension=1] coordinates{(0,0) (2.5,1.5) (0,3)};

\draw[draw=white,fill=white] 
    plot[smooth, tension=1] coordinates{(0,0) (1,1.5) (0,3)};

 \draw[white,thick]
(0,0) -- (0,3);

\filldraw (0,3) circle (3pt)node[anchor=south]{$x$};
\filldraw (0,0) circle (3pt)node[anchor=north]{$y$};
\filldraw (-1,4) circle (3pt);
\filldraw (-1,1.5) circle (3pt);
\filldraw (-1,-1) circle (3pt);
\filldraw (-2,3) circle (3pt);
\filldraw (-2,0) circle (3pt);

\filldraw (-.6,.8) circle (0pt)node[anchor=west]{$g$};

 \draw[black,thick]
(-2,0) -- (-1,1.5) -- (0,3) -- (-1,4) -- (-2,3) -- (0,0) -- (-1,-1) -- (-2,0) -- (0,0);

 \draw[black,thick]
(0,3) -- (-2,3);

 \draw[black,thick]
(-1,4) -- (-1,1.5) -- (-1,-1);

\draw[thick] plot[smooth, tension=1] coordinates{(0,0) (1,1.5)
(0,3)}; \draw[thick] plot[smooth, tension=1] coordinates{(0,0)
(2.5,1.5) (0,3)};

  \node [rectangle, draw=white, fill=white] (b) at (-1,-2) {$(i)$};

\filldraw (1.7,1) circle (0pt)node[anchor=south]{$F_2$};

\filldraw (10,3) circle (3pt)node[anchor=south]{$x$};
\filldraw (10,0) circle (3pt)node[anchor=north]{$y$};
\filldraw (9,4) circle (3pt);
\filldraw (9,1.5) circle (3pt);
\filldraw (9,-1) circle (3pt);
\filldraw (8,3) circle (3pt);
\filldraw (8,0) circle (3pt);
\filldraw (12,0) circle (3pt);
\filldraw (12,3) circle (3pt);

 \draw[black,thick]
(8,0) -- (9,1.5) -- (10,3) -- (9,4) -- (8,3) -- (9,1.5);

 \draw[black,thick]
(10,0) -- (9,-1) -- (8,0) -- (10,0);

 \draw[black,thick]
(10,0) -- (12,3) -- (12,0) -- (10,0) -- (10,3) -- (8,3);

 \draw[black,thick]
(9,4) -- (9,1.5) -- (9,-1);

 \draw[black,thick]
(12,0) -- (10,3) -- (12,3);

  \node [rectangle, draw=white, fill=white] (b) at (9,-2) {$(ii)$};

\filldraw (1.7,1) circle (0pt)node[anchor=south]{$F_2$};
\end{tikzpicture}
\caption{The graphs $G$ and $H_3$ in the subcase $G_1=H_2$ of Subcase 3.2.} \label{fig:H_2new}
\end{figure}

Hence we may assume that $G_1\not\in \{K_5^-,H_1,H_2\}$. Then  $G_1$ has two admissible nodes by Theorem \ref{thm:nearly3con}.
It can be seen that any admissible node of $G_1$ which is distinct from $x,y$ is an admissible node of $G$, so we may suppose that $x$ and $y$ are the only admissible nodes of $G_1$. Then both $x$ and $y$ have degree three in $G_1$. Let $a,b,x$ be the neighbours of $y$ in $G_1$. Since $x$ has degree two in $G_1-y$, the admissible 1-reduction at $y$ constructs a new graph $G_1'$ from $G_1$ by deleting $y$ and then adding a new edge from $x$ to either $a$ or $b$, say $b$.
We use the admissible reduction of $G_2$ to $G_2''$ described in the first paragraph of this subcase. Let $G'$ be the graph obtained from $G$ by performing an edge-reduction which deletes the edge $ay$ and then contracts the edge $f\in E(F_2)$. Then
$G'=G_1'*_1 H_2 *_1 G_2''$ and hence $G'$ is an $\M_{2,2}^*$-circuit,  see Figure \ref{fig:final}. Thus the edge-reduction which transforms $G$ to $G'$ is admissible.
\end{proof}

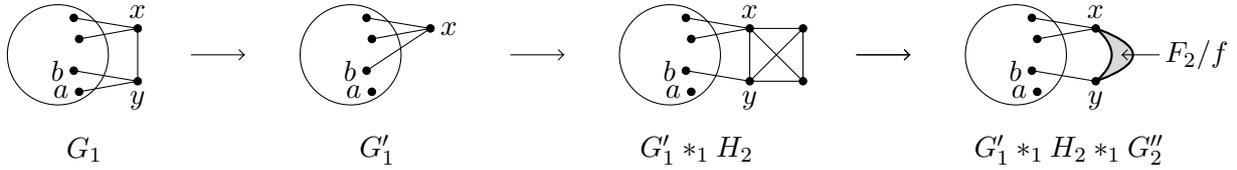
\begin{figure}
\begin{tikzpicture}[scale=0.7]

\draw (-4,10) circle (27pt);
\draw (1.5,10) circle (27pt);
\draw (7.5,10) circle (27pt);
\draw (14,10) circle (27pt);




\filldraw (-2.5,10.5) circle (2pt) node[anchor=south]{$x$};
\filldraw (-2.5,9.5) circle (2pt) node[anchor=north]{$y$};
\filldraw (-3.7,9.7) circle (2pt) node[anchor=east]{$b$};
\filldraw (-3.6,9.3) circle (2pt) node[anchor=east]{$a$};
\filldraw (-3.6,10.3) circle (2pt);
\filldraw (-3.7,10.7) circle (2pt);

\draw[black]
(-3.7,10.7) -- (-2.5,10.5) -- (-2.5,9.5);

\draw[black]
(-3.6,10.3) -- (-2.5,10.5);

\draw[black]
(-3.7,9.7) -- (-2.5,9.5) -- (-3.6,9.3);

\draw[black]
(-1.5,10) -- (-.5,10) -- (-.6,9.9);

\draw[black]
(-.5,10) -- (-.6,10.1);

  \node [rectangle, draw=white, fill=white] (b) at (-3.5,8.2) {$G_1$};


\filldraw (3,10.5) circle (2pt) node[anchor=west]{$x$};
\filldraw (1.8,9.7) circle (2pt) node[anchor=east]{$b$};
\filldraw (1.9,9.3) circle (2pt) node[anchor=east]{$a$};
\filldraw (1.9,10.3) circle (2pt);
\filldraw (1.8,10.7) circle (2pt);

\draw[black]
(1.8,10.7) -- (3,10.5);

\draw[black]
(3,10.5) -- (1.8,9.7);

\draw[black]
(1.9,10.3) -- (3,10.5);

\draw[black]
(4.5,10) -- (5.5,10) -- (5.4,9.9);

\draw[black]
(5.5,10) -- (5.4,10.1);

  \node [rectangle, draw=white, fill=white] (b) at (2,8.2) {$G_1'$};


\filldraw (9,10.5) circle (2pt) node[anchor=south]{$x$};
\filldraw (9,9.5) circle (2pt) node[anchor=north]{$y$};
\filldraw (10,10.5) circle (2pt);
\filldraw (10,9.5) circle (2pt);
\filldraw (7.8,9.7) circle (2pt) node[anchor=east]{$b$};
\filldraw (7.9,9.3) circle (2pt) node[anchor=east]{$a$};
\filldraw (7.9,10.3) circle (2pt);
\filldraw (7.8,10.7) circle (2pt);

\draw[black]
(7.8,10.7) -- (9,10.5) -- (10,10.5) -- (10,9.5) -- (9,9.5) -- (10,10.5);

\draw[black]
(10,9.5) -- (9,10.5) --  (9,9.5) -- (7.8,9.7);

\draw[black]
(7.9,10.3) -- (9,10.5);

\draw[black]
(11,10) -- (12,10) -- (11.9,9.9);

\draw[black]
(12,10) -- (11.9,10.1);

  \node [rectangle, draw=white, fill=white] (b) at (8,8.2) {$G_1'*_1 H_2$};


\draw[draw=gray!30!white,fill=gray!30!white] 
    plot[smooth, tension=1] coordinates{(15.5,10.5)
(16.2,10) (15.5,9.5)}  -- 
    plot[smooth, tension=1] coordinates{(15.5,10.5)
(16.2,10) (15.5,9.5)};

\draw[draw=white,fill=white] 
    plot[smooth, tension=1] coordinates{(15.5,10.5) (15.8,10)
(15.5,9.5)};

 \draw[white,thick]
(15.5,10.5) -- (15.5,9.5);

\filldraw (15.5,10.5) circle (2pt) node[anchor=south]{$x$};
\filldraw (15.5,9.5) circle (2pt) node[anchor=north]{$y$};
\filldraw (14.3,9.7) circle (2pt) node[anchor=east]{$b$};
\filldraw (14.4,9.3) circle (2pt) node[anchor=east]{$a$};
\filldraw (14.4,10.3) circle (2pt);
\filldraw (14.3,10.7) circle (2pt);

\draw[black]
(14.3,10.7) -- (15.5,10.5);

\draw[black]
(15.5,9.5) -- (14.3,9.7);

\draw[black]
(14.4,10.3) -- (15.5,10.5);

\draw[black]
(11,10) -- (12,10) -- (11.9,9.9);

\draw[black]
(12,10) -- (11.9,10.1);

  \node [rectangle, draw=white, fill=white] (b) at (15,8.2) {$G_1'*_1 H_2*_1G_2''$};

\draw[thick] plot[smooth, tension=1] coordinates{(15.5,10.5) (15.8,10)
(15.5,9.5)}; \draw[thick] plot[smooth, tension=1] coordinates{(15.5,10.5)
(16.2,10) (15.5,9.5)};

\node [rectangle, draw=white, fill=white] (b) at (17.4,10) {$F_2/f$};

\draw[black]
(16.7,10) -- (16,10) -- (16.1,10.1);

\draw[black]
(16,10) -- (16.1,9.9);

\end{tikzpicture}
\caption{The subcase of Subcase 3.2 when $G_1\not\in \{K_5^-,H_1,H_2\}$.} \label{fig:final}
\end{figure}

To see that the $K_4^-$-extension operation is needed in Theorem \ref{strongadmiss}, observe that we can construct graphs which do not admit admissible edge-reductions as follows. Take any $\M_{2,2}^*$-circuit $H$, and apply the $K_4^-$-extension to every single edge of $G$. The resulting graph $G$ has two types of edges. Those edges with no end-vertices in $H$ are contained in two triangles so any edge-reduction which contracts such an edge results in a non-simple graph. The remaining edges, those with exactly one end-vertex in $H$, are not admissible either since any edge-reduction which contracts such an edge results in a graph containing a vertex of degree two.

\section{Rigidity and Stress Matrices}\label{sec:rig}

Let  $G=(V,E)$ be a graph with $V=\{v_1,\dots,v_n\}$. We will
consider realisations of $G$ on the unit cylinder
$\mathcal{Y}=\{(x,y,z)\in \mathbb{R}^3:x^2+y^2=1\}$.\footnote{For
the purposes of (global) rigidity there is no loss in generality in
assuming our cylinder has unit radius and is centred on the
$z$-axis.} A \emph{framework} $(G,p)$ on $\mathcal{Y}$ is an ordered
pair consisting of a graph $G$ and a realisation $p$ such that
$p(v_i)\in \mathcal{Y}$ for all $v_i\in V$.

Two frameworks $(G,p)$ and $(G,q)$ on $\mathcal{Y}$ are \emph{equivalent} if $\|p(v_i)-p(v_j)\|=\|q(v_i)-q(v_j)\|$ for all edges $v_iv_j\in E$. Moreover $(G,p)$ and $(G,q)$ are \emph{congruent} if $\|p(v_i)-p(v_j)\|=\|q(v_i)-q(v_j)\|$ for all pairs of vertices $v_i,v_j\in V$.
The framework $(G,p)$ is \emph{globally rigid} on $\mathcal{Y}$ if every equivalent framework $(G,q)$ on $\mathcal{Y}$ is congruent to $(G,p)$.
It is \emph{rigid} on $\mathcal{Y}$ if there exists an $\epsilon>0$ such that every framework $(G,q)$ on $\mathcal{Y}$ which is equivalent to $(G,p)$, and has $\| p(v_i)-q(v_i)\|<\epsilon$ for all $1\leq i \leq n$, is congruent to $(G,p)$.
It is \emph{generic} on $\mathcal{Y}$ if $\td[\mathbb{Q}(p):\mathbb{Q}]=2n$.


The \emph{rigidity matrix} $R_{\CYL}(G,p)$ of a framework $(G,p)$ on $\mathcal{Y}$ is the $(|E|+|V|)\times 3|V|$ matrix
\[ R_{\CYL}(G,p)=\begin{pmatrix}R(G,p)\\ S(G,p) \end{pmatrix}\]
where:
$R(G,p)$ has rows indexed by $E$ and 3-tuples of columns indexed by $V$ in which, for $e=v_iv_j\in E$, the submatrices in row $e$ and columns $v_i$ and $v_j$ are $p(v_i)-p(v_j)$ and $p(v_j)-p(v_i)$, respectively, and all other entries are zero; $S(G,p)$ has rows indexed by $V$ and 3-tuples of columns indexed by $V$ in which, for $v_i\in V$, the submatrix in row $v_i$ and column $v_i$ is $\bar p(v_i)=(x_i,y_i,0)$ when $p(v_i)=(x_i,y_i,z_i)$. We refer to the vectors in the kernel of $R_{\CYL}(G,p)$ as \emph{infinitesimal flexes} of $(G,p)$. We will consider an {infinitesimal flex} as a map $s:V\rightarrow \mathbb{R}^3$ such that $s(v_i)$ is tangential to $\mathcal{Y}$ at $p(v_i)$ for all $v_i\in V$ and $(p(v_j)-p(v_i))\cdot (s(v_j)-s(v_i))=0$ for all $v_jv_i\in E$. The kernel of $R_{\CYL}(G,p)$ will always contain two linearly independent flexes corresponding to a translation along the axis of the cylinder and a rotation about the   
same axis. This implies that  $\rank R_{\CYL}(G,p)\leq 3n-2$. We say that 
$(G,p)$ is \emph{infinitesimally rigid} on $\mathcal{Y}$ if  $\rank R_{\CYL}(G,p)= 3n-2$.
It was shown in \cite{NOP} that a generic framework $(G,p)$ on $\Y$ is rigid if and only if it is a complete graph on at most three vertices or is infinitesimally rigid.

An \emph{equilibrium stress} for a framework $(G,p)$ on $\Y$ is a pair $(\omega,\lambda)$, where $\omega:E\to \bR$ and $\lambda:V\to \bR$ and $(\omega,\lambda)$ belongs to the cokernel of $R_{\CYL}(G,p)$. Thus
$(\omega,\lambda)$ is an equilibrium stress for $(G,p)$ on $\Y$ if and only if
 \begin{equation}\label{eq:stressdefn}
 \sum_{j=1}^n \omega_{ij}(p(v_i)-p(v_j)) + \lambda_i \bar p(v_i)=0 \mbox{ for all $1\leq i \leq n$},
\end{equation}
where $\omega_{ij}$ is taken to be equal to $\omega_e$ if $e=v_iv_j\in E$ and to be equal to $0$ if $v_iv_j\not\in E$.

Given a stress $(\omega,\lambda)$ for a framework $(G,p)$ on $\Y$ we define $\Omega=\Omega(\omega)$ to be the $n\times n$ symmetric matrix with off-diagonal entries $-\omega_{ij}$ and diagonal entries $\sum_j \omega_{ij}$, and
$\Lambda=\Lambda(\lambda)$ to be the $n\times n$ diagonal matrix with diagonal entries $\lambda_1,\lambda_2,\ldots,\lambda_n$.
The {\em stress matrix} associated to $(\omega,\lambda)$ is the $3n\times 3n$ symmetric matrix

$$\Omega_{\CYL}(\omega,\lambda)=\begin{bmatrix} \Omega + \Lambda & 0 & 0\\ 0 & \Omega + \Lambda & 0\\ 0 & 0 & \Omega \end{bmatrix}.$$

It follows from (\ref{eq:stressdefn}) that $(x_1,x_2,\dots,x_n),
(y_1,y_2,\dots,y_n)$ belong to the  cokernel of $\Omega+\Lambda$ and
$(z_1,z_2,\dots,z_n), (1,1,\dots,1)$ belong to the cokernel of
$\Omega$. Hence $\rank (\Omega+\Lambda)\leq n-2$, $\rank \Omega \leq
n-2$ and $\rank \Omega_{\CYL}(\omega,\lambda)\leq 3n-6$. We will say
that $(\omega,\lambda)$ {\em has maximum rank} when $\rank \Omega_{\CYL}(\omega,\lambda)= 3n-6$.

\section{Maximum rank equilibrium stresses}\label{sec:vsplit}

We will show that every generic realisation of an $\M_{2,2}^*$-circuit on the cylinder has a maximum rank equilibrium stress. We do this by showing that both of the recursive operations used in Theorem \ref{thm:strongrecurse} preserve this property.

We will need the following elementary tool from linear algebra. Suppose $M=\begin{bmatrix} A & B\\ C& D\end{bmatrix}$ is a block matrix and $A$ is invertible. Then the \emph{Schur complement} of $A$ in $M$ is the matrix $F=D-CA^{-1}B$ and we have 
\begin{equation} \rank M = \rank A + \rank F.\label{eqn:schur}\end{equation}

We will also need the following extension of \cite[Lemmas 9 and 10]{J&N}.

\begin{lem}\label{lem:generic}
Suppose $(G,p)$ is an infinitesimally rigid framework on the
cylinder $\Y$. Then $(G,q)$ is infinitesimally rigid on
$\Y$ for all generic $(G,q)$ on $\Y$. Moreover, if
$(\omega,\lambda)$ is an equilibrium stress for $(G,p)$ on $\Y$ with
$\rank \Omega_\CYL(\omega,\lambda)=3n-6$, then $(G,q)$ has an
equilibrium stress $(\omega',\lambda')$ on $\Y$ with $\rank
\Omega_{\CYL}(\omega',\lambda')=3n-6$ for all generic $(G,q)$ on $\Y$.
\end{lem}

\begin{proof}
The first assertion follows from the fact that the $\rank R_{\CYL}(G,p)$ will be maximised at any generic framework $(G,p)$.
We adapt the proof technique of Connelly and Whiteley \cite[Theorem 5]{CW} to prove the second assertion.
Choose an arbitrary infinitesimally rigid framework $(G,p')$ on $\Y$.
Since the entries in $R_{\CYL}(G,p')$ are polynomials in $p'$ and the space of equilibrium stresses of $(G,p')$ is the cokernel of $R_{\CYL}(G,p')$, each equilibrium stress of $(G,p')$ can be expressed as a pair of rational functions $(\omega(p',t),\lambda(p',t))$ of $p'$ and $t$, where $t$ is a vector of $m-2n+2$ indeterminates. (To see this we can imagine solving $R_{\CYL}(G,p)^Ts=0$ by Gaussian elimination. The resulting matrix will have $\rank R_{\CYL}(G,p)=3n-2$ leading entries. We can express a general solution by considering the stresses corresponding to the  $(m+n)-(3n-2)=m-2n+2$ columns which do not contain leading entries as indeterminates and then solving for the stresses corresponding to the remaining columns.)  
This implies that the entries in the corresponding stress matrix $\Omega_{\CYL}(\omega(p',t),\lambda(p',t))$ will also be rational functions of $p'$ and $t$. Hence the rank of
$\Omega_{\CYL}(\omega(p',t),\lambda(p',t))$ will be maximised whenever $(G,p')$ is generic on $\Y$ and $t$ is algebraically independent over $\bQ(p')$. Hence, for any generic  $(G,q)$ on $\Y$, $(G,q)$ is infinitesimally rigid on $\Y$  and we can choose $t\in \bR^{m-2n+2}$ such that
 $\rank \Omega_{\Y}(\omega(q,t),\lambda(q,t))=3n-6$.
\end{proof}

We first show that the $K_4^-$-extension operation preserves the property of having a maximum rank equilibrium stress.

\begin{thm}\label{lem:typea1}
Suppose $G=(V,E)$ and $G_1=(V_1,E_1)$ are graphs, $v_2v_3\in E_1$, $G_1- v_2v_3$ is rigid and $G$ is a $K_4^-$-extension
of $G_1$ on the edge $v_2v_3$. Let $(G,p)$ be a generic realisation of $G$ on the unit
cylinder $\Y$ and let $p_1$ be the restriction of $p$ to $G_1$.
Suppose $(G_1,p_1)$ has a maximum rank equilibrium stress on $\Y$.
Then $(G,p)$ has a maximum rank equilibrium stress on $\Y$.
\end{thm}

\begin{proof}
Let $V_1=\{v_2,v_3,\dots,v_n\}$ and suppose that $G$ is constructed from $G_1$ by deleting $v_2v_3$, adding two new vertices $v_0,v_1$ and five new edges $v_0v_1,v_0v_2,v_0v_3,v_1v_2,v_1v_3$. We may use the isometries of $\Y$ to move $(G,p_1)$ so that $p_1(v_2)=(0,1,0)$. Let $p_1(v_i)=(x_i,y_i,z_i)$ for $3\leq i \leq n$.
Define $q:V\rightarrow \mathbb{R}^3$ by putting $q(v)=p_1(v)$ for all $v\in V_1$ and choosing $q(v_0)=(0,1,1)$ and $q(v_1)=(-x_3,y_3,z_3)$.
(We choose these values for $q$ so that the rows of $R_\CYL(G+v_2v_3,q)$ labelled by the vertices and edges of the subgraph $H$ of $G+v_2v_3$  induced by $\{v_0,v_1,v_2,v_3\}$, are dependent. This will enable us to construct an equilibrium stress for $(G,q)$ by combining equilibrium stresses for $(G_1,p_1)$ and $(H,q|_H)$ in such a way that the net stress on $v_2v_3$ is zero.)

We first show that $(G,q)$ is infinitesimally rigid on $\Y$. This follows from the facts that
$G_1-v_2v_3$ is rigid and that $(G-v_0v_1,q)$ can be constructed from $(G_1-v_2v_3,p_1)$ by adding $v_0$ and $v_1$ as vertices of degree two at points which do not lie on the lines joining their two neighbours $v_2,v_3$.

Now suppose $(\omega',\lambda')$ is a  maximum rank equilibrium
stress for $(G_1,p_1)$.
Since $G_1-v_2v_3$ is rigid, we may
suppose $(\omega',\lambda')$ has been chosen so that the stress
value on $v_2v_3$ is non-zero and hence we may scale
$(\omega',\lambda')$ so that $\omega'_{23}=-1$. We may combine
$(\omega',\lambda')$ with the unique equilibrium stress for
$(H,q|_H)$ which has stress value one on $v_2v_3$ to obtain the
equilibrium stress $(\omega,\lambda)$ for $(G,q)$ on the unit
cylinder defined by  $\omega_f=\omega'_f$ for all $f\in E_1-v_2v_3$,
$\lambda_i=\lambda'_i$ for all $4\leq i \leq n$, $\omega_{23}=0$,
$\omega_{12}=1$, $\omega_{02}=-2z_3$,
$\omega_{13}=\frac{1}{2y_3(z_3-1)}$,
$\omega_{03}=-\frac{z_3}{z_3-1}=\omega_{01}$,
$\lambda_2=\lambda'_2+2y_3-2$,
$\lambda_3=\lambda'_3+\frac{y_3-1}{y_3(z_3-1)}$,
$\lambda_1=\frac{y_3-1}{y_3(z_3-1)}$, and
$\lambda_0=-\frac{2z_3(y_3-1)}{z_3-1}$.

We have $\Omega(\omega)=\begin{bmatrix} A & B\\ B^T& D\end{bmatrix}$ where
$$A=\begin{bmatrix}  \omega_{03}+\omega_{02}+\omega_{01} & -\omega_{01}  \\  -\omega_{01} & \omega_{13}+\omega_{12}+\omega_{01} \end{bmatrix},\; B= \begin{bmatrix}  -\omega_{02} & -\omega_{03} & 0 & \dots \\ -\omega_{12} & -\omega_{13} & 0 & \dots  \end{bmatrix}, 
$$
and
$$D=\begin{bmatrix} \sum_{j\geq 4} \omega_{2j}+\omega_{02} +\omega_{12} & 0 & -\omega_{24} & \dots \\ 0 & \sum_{j\geq 4} \omega_{3j} +w_{03} +w_{13} & -\omega_{34} & \dots \\ \vdots & \vdots & \vdots \end{bmatrix} \,.$$
We can now substitute the values for $\omega_{01},\omega_{02},\omega_{03},\omega_{12},\omega_{13}$ into $A,B$ to obtain
$$A=\begin{bmatrix} -2z_3-\frac{2z_3}{z_3-1} & \frac{z_3}{z_3-1}\\
\frac{z_3}{z_3-1} & 1+\frac{1}{2y_3(z_3-1)}-\frac{z_3}{z_3-1}
 \end{bmatrix}\,\mbox{ and }\,B= \begin{bmatrix} 2z_3 & \frac{z_3}{z_3-1} & 0 & \dots \\ -1 & \frac{-1}{2y_3(z_3-1)} & 0 & \dots  \end{bmatrix}.$$
Since $\{y_3,z_3\}$ is algebraically independent over $\mathbb{Q}$, $A$ is invertible and a matrix calculation gives
$$B^TA^{-1}B=\begin{bmatrix}-2z_3+2 & -1 & 0 & \dots \\
-1 &  -\frac{2y_3-1}{2y_3(z_3-1)}
 & 0 & \dots\\
0 & 0 & 0 & \dots \\ \vdots & & \vdots \end{bmatrix}. $$
We may now deduce that
$$D-B^TA^{-1}B=\begin{bmatrix} \sum_{j\geq 4} \omega_{2j}-1 & 1 & -\omega_{24} & \dots \\ 1 & \sum_{j\geq 4} \omega_{3j} -1 & -\omega_{34} & \dots \\ \vdots & \vdots & \vdots \end{bmatrix}=\Omega(\omega').$$
Equation (\ref{eqn:schur}) now gives $\rank \Omega(\omega)= \rank A+\rank \Omega(\omega')=|V(G)|-2$.

We may use a similar calculation to deduce that $\rank [\Omega(\omega)+\Lambda(\lambda)]=|V(G)|-2$. Hence $\rank \Omega_{\CYL}(\omega,\lambda)=3|V(G)|-6$.  The result now follows from Lemma \ref{lem:generic}.
\end{proof}

We chose the values of $q(v_0)$ and $q(v_1)$ in the above proof so that the framework  $(H,q|_H)$  would have a nowhere zero equilibrium stress.  We could also have accomplished this by putting $q(v_0)$ and $q(v_1)$ in the same plane as $q(v_2)$ and $q(v_3)$, but such a choice  would have resulted in the matrix $A$ being singular.

We next consider the generalised vertex splitting operation. In \cite[Theorem $5.2$]{JKN} it was proved that the standard vertex splitting operation, with the additional assumption that the new
graph is rigid when we delete the \emph{bridging edge} (i.e. the edge joining the two copies of the split vertex), preserves generic global rigidity on the cylinder. We will need a similar result for
generalised vertex splitting and maximum rank equilibrium stresses. Our proof technique is based on a proof by Connelly \cite{Cnotes} that the standard vertex splitting operation preserves the property
of having a maximum rank equilibrium stress in Euclidean space. We first give a variant of  \cite[Lemma 28]{Cnotes}.

\begin{lem}\label{lem:matrixlimit} Let $W_t$ be an $n\times n$ matrix whose entries are continuous functions of $t$ for all $t\in (0,\delta)$ and $X$ be a non-singular $n\times n$ matrix. Suppose that $X^TW_tX=\begin{bmatrix}A_t&B_t\\C_t&D_t\end{bmatrix}$, that $B_t,C_t,D_t$ tend to finite limits $B,C,D$, respectively, as $t$ approaches $0$, and that $A_t=(a_t)$ is a $1\times 1$ matrix with $\lim_{t\to 0}|a_t|=\infty$. Then $\rank W_t\geq \rank D +1$ for all $t$ sufficiently close to $0$.
\end{lem}
\begin{proof}
Equation (\ref{eqn:schur}) and the hypothesis that $X$ is non-singular  give
$$\rank W_t=\rank X^TW_tX=\rank A_t+\rank (D_t-a_t^{-1}C_tB_t).$$
Since $\lim_{t\to 0} |a_t|=\infty$ and $\lim_{t\to 0}D_t=D$, we have $\rank A_t=1$
and $\rank (D_t-a_t^{-1}C_tB_t)\geq \rank D$
when $t$ is sufficiently
close to zero.
\end{proof}

We will also need the following result about frameworks with two coincident points. Let $G$ be a graph with two distinguished vertices $u$ and $v$. A framework $(G,p)$ is \emph{$uv$-coincident} if $p(u)=p(v)$. A $uv$-coincident framework is \emph{$uv$-generic} if $(G-u,p')$ is generic, where $p'$ is the restriction of $p$ to $G-u$.

\begin{thm}{\cite[Theorem 18]{JKN}}\label{thm:uvrigid}
Let $u$ and $v$ be distinct vertices of a graph $G$ and let $(G,p)$ be a $uv$-generic, $uv$-coincident realisation of $G$ on the cylinder. Then $(G,p)$ is infinitesimally rigid if and only if the graphs $G-uv$ and $G/uv$ are both rigid on the cylinder.
\end{thm}

We need one more result to show that the generalised vertex split operation preserves the property of
having a maximum rank equilibrium stress when applied to $\M_{2,2}^*$-circuits on $\Y$.

\begin{lem}\label{lem:expand}
Let $(G, p)$ be a generic realisation of an $\M_{2,2}^*$-circuit on $\Y$ and $(\omega, \lambda)$ be a non-zero equilibrium stress
for $(G,p)$. Then $\lambda_i\neq 0$ for all $v_i\in V$.
\end{lem}

We will delay the proof of this lemma until Section \ref{sec:vfree} as its proof is fairly long and will be a distraction from our next result.

\begin{thm}\label{thm:vsplitmax}
Suppose that $(G,p)$ is a generic realisation of an
$\M_{2,2}^*$-circuit on $\Y$ and that
$(\omega,\lambda)$ is a maximum rank equilibrium stress for $(G,p)$.
Let $\hat G$ be obtained from $G$ by a generalised vertex splitting
operation and suppose that $\hat G$ is an $\M_{2,2}^*$-circuit on
the cylinder. Then there exists a realisation $(\hat G,q)$  on $\Y$
which is infinitesimally rigid and has a maximum rank 
equilibrium stress.
\end{thm}

\begin{proof}
Suppose that $V(G)=\{v_1,v_2,\dots,v_n\}$ and that $\hat G$ is obtained from $G$ by choosing the vertex $v_1$ with neighbours $v_2,v_3,\ldots,v_m$, deleting the edges
$v_1v_2,v_1v_3,\ldots,v_1v_{k}$ and then adding a new vertex $v_0$
and new edges $v_0v_1,v_0v_2,\ldots,v_0v_{k}$ and $v_0v_\ell$  for some $\ell\geq k+1$.

Let $(\hat G,\hat p)$ be the $v_0v_1$-coincident $v_0v_1$-generic
framework with $\hat p|_G=p$ and $\hat p(v_0)=p(v_1)$. Then  $(\hat
G-v_0v_1,\hat p)$ is infinitesimally rigid by Theorem \ref{thm:uvrigid} (since $\hat G-v_0v_1$ and $(\hat
G-v_0v_1)/v_0v_1=G+v_0v_\ell$ are both rigid on the cylinder). Hence $\rank R_\CYL(\hat G-v_0v_1,\hat p)=3|V(\hat G)|-2$. Since $\hat G$ is an $\M_{2,2}^*$-circuit, this implies that the rows of $\hat R_{\CYL}(\hat G-v_0v_1,\hat p)$ are linearly independent.

Let $a=\sum_{j=2}^k\omega_{1j}(p(v_1)-p(v_j))$ and
$b=\sum_{j=k+1}^m\omega_{1j}(p(v_1)-p(v_j))$. Since
$(\omega,\lambda)$ is an equilibrium stress, we have
$a+b+\lambda_1\bar p(v_1)=0$ and since $\lambda_1\neq 0$ by Lemma
\ref{lem:expand}, $\bar p(v_1)\in \span \{a,b\}$.

We next show that $\dim \span \{a,b\}=2$. Suppose to the contrary that $\dim \span \{a,b\}=1$. Then $a$ and $b$ are both scalar
multiples of $\bar p(v_1)$ so $a+\mu \bar p(v_1)=0=b+\nu \bar p(v_1)$ for some $\mu,\nu\in \bR$. We can now define a non-zero equilibrium
stress $(\hat \omega,\hat \lambda)$ for $(\hat G-v_0v_1,\hat p)$ as
follows:
 $\hat \omega_{0j}=\omega_{1j}$ for $2\leq j \leq k$ and otherwise $\hat \omega_{0j}=0$, 
 $\hat \omega_{1j}=\omega_{1j}$ for $k+1\leq j\leq m$ and otherwise $\hat \omega_{1j}=0$,
 $\hat \omega_{ij}=\omega_{ij}$ for all $i,j\geq 2$;
 $\hat \lambda_0=\mu$,
$\hat \lambda_1=\nu$,
$\hat \lambda_i=\lambda_i$ for all $i\geq 2$. 
This contradicts
the fact that the rows of $R_\CYL(\hat G-v_0v_1,\hat p)$ are
linearly independent. Hence $\dim \span \{a,b\}=2$.

Let $\P$ be the plane which passes through $p(v_1)$ and whose normal
belongs to $\span\{a,b\}^\perp$. Since $\bar p(v_1)\in
\span\{a,b\}$, $\P$ is not tangential to $\Y$ at $p(v_1)$ and hence
$\P$ intersects $\Y$ in a curve $\C$ which passes through $p(v_1)$.
Since $(\hat G-v_0v_1,\hat p)$ is infinitesimally rigid, we have
$\rank R_\CYL(\hat G-v_0v_1,q)=3|V(\hat G)|-2$ for all $q$
sufficiently close to $\hat p$ on $\Y$. Hence we may choose a simple path $Q:[0,1]\to \C$
such that $Q(0)=p(v_1)$ and such that the realisation $(\hat G,q^t)$ with $q^t|_G=p$ and
$q^t(v_0)=Q(t)$ is infinitesimally rigid for all $t\in [0,1]$.
Then $c^t=q^t(v_0)-p(v_1)$ satisfies $\span \{a,b\}=\span \{c^t,\bar
p(v_1)\}$ for all $t\in (0,1]$.

Let $R_t$ be obtained from $R_\CYL(\hat G,\hat p)$
by replacing the zero row indexed by $v_0v_1$ by the row with $c^t$
and $-c^t$ in the $v_0$ and $v_1$ columns and zeros elsewhere. The
choice of $c^t$ tells us that, for all $t\in (0,1]$, there exist unique scalars $\omega^t_{01},
\bar\omega^t_{01}, \lambda^t_0, \lambda^t_1$ such
that
\begin{equation}\label{eq:v_0}
 a+\lambda^t_0\bar p(v_1)
+\omega^t_{01}c^t=0 \end{equation} and
\begin{equation}\label{eq:v_1hat}
b+\lambda^t_1\bar p(v_1) -\bar \omega^t_{01}c^t=0.
\end{equation}
 Since $(\omega,\lambda)$ is an equilibrium
stress for $(G,p)$ we have
\begin{equation}\label{eq:v_1}
a+b+\lambda_1\bar p(v_1)=0. \end{equation}
Hence $(\lambda^t_0+\lambda^t_1)\bar p(v_1)+(
\omega^t_{01}-\bar\omega^t_{01})c^t=\lambda_1\bar p(v_1)$. It
follows that $\lambda_1=\lambda_0^t+\lambda_1^t$ and
$\omega^t_{01}=\bar\omega^t_{01}$.

We can extend $\omega^t_{01},\lambda^t_0,\lambda^t_1$ to a vector $(\omega^t,\lambda^t) \in \coker R_t$ as follows: 
$\omega^t_{0j}
=\omega_{1j}$ for $2\leq j\leq k$ and $\omega^t_{0j} =0$ for $j>
k$,
$\omega^t_{1j} =\omega_{1j}$ for 
$k+1\leq j\leq m$ and $\omega^t_{1j} =0$ for $2\leq j\leq k$ or
$j>m$, $\omega^t_{ij} =\omega_{ij}$ for $i,j\geq 2$;
$\lambda^t_i=\lambda_i$ for $i\geq 2$.
Since $\rank R_\CYL(\hat G-v_0v_1,\hat p)=3|V(\hat G)|-2$, the rows of $R_t$ indexed by $E(\hat G)-v_0v_1$  are linearly independent. The fact that $(\omega^t,\lambda^t)$ is a
non-zero vector in $\coker R_t$ now gives $
\omega^t_{01}\neq 0$.

The matrix $\Omega(\omega^t)$ defined by
$(\omega^t,\lambda^t)$ has the form
$$\Omega(\omega^t)=
\begin{bmatrix}
\omega^t_{01}+\sum_{j\geq 2}\omega^t_{0j} & -\omega^t_{01} & -\omega^t_{02}    & \dots\\[2mm]
-\omega^t_{01} & \omega^t_{01}+\sum_{j\geq 2}\omega^t_{1j} & -\omega^t_{12}  & \dots\\[2mm]
-\omega^t_{02} & -\omega^t_{12} &  \sum_{j\geq 2}\omega_{2j}  & \dots\\[1mm]
\vdots & \vdots & \vdots & \vdots   \\
 \end{bmatrix}.
$$
Let $X=\begin{bmatrix} Y & 0\\0& I\end{bmatrix}$ be the $(n+1)\times (n+1)$ block matrix with $Y=\begin{bmatrix}1 & 1\\0 & 1\end{bmatrix}$.
We may use the fact
that $\omega^t_{1j}+\omega^t_{0j}=\omega_{1j}$
for all $2\leq j\leq k$  to obtain
$$
X^T\Omega(\omega^t)X=
\begin{bmatrix}
\omega^t_{01}+\sum_{j\geq 2}\omega^t_{0j} & \sum_{j\geq 2}\omega^t_{0j} & -\omega^t_{02}    & \dots\\[2mm]
\sum_{j\geq 2}\omega^t_{0j} & \sum_{j\geq 2}\omega_{1j} & -\omega_{12}  & \dots\\[2mm]
-\omega^t_{02} & -\omega_{12} &  \sum_{j\geq 2}\omega_{2j}  & \dots\\[1mm]
\vdots & \vdots & \vdots & \vdots   \\
 \end{bmatrix}
 =
\begin{bmatrix}
A_t & B\\[1mm]
B^T & \Omega(\omega)
 \end{bmatrix}
$$
 where $A_t=(\omega^t_{01}+\sum_{j\geq
2}\omega^t_{0j})$ and $B=(\sum_{j\geq
2}\omega^t_{0j},-\omega^t_{02}, \ldots,-\omega^t_{0n})$. (Note that the entries in $B$ do not change with $t$.)
 We will show that the entry in $A_t$ is unbounded as $t\to 0$.

Let $\bar p(v_1)^\perp$ be a unit vector in $\span\{a,b\}$ which is
orthogonal to $\bar p(v_1)$. Then Equation (\ref{eq:v_0}) gives
$a\cdot \bar p(v_1)^\perp+\omega^t_{01}c^t\cdot \bar
p(v_1)^\perp=0$. Since $c^t$ and $\bar p(v_1)$ are linearly
independent, we have $c^t\cdot \bar p(v_1)^\perp\neq 0$ and
\begin{eqnarray}\label{eqn:wo1} \omega^t_{01}=-\frac{a\cdot \bar p(v_1)^\perp}{c^t\cdot \bar
p(v_1)^\perp}. \end{eqnarray}
As $t\to 0$, $q(v_0)$ approaches $p(v_1)$ on
$\C$, and hence $c^t$ will approach $(0,0,0)$. Hence $|\omega^t_{01}|$ will become
arbitrarily large. On the other hand $|w^t_{0j}|\leq |w_{1j}|$ for all $j\geq 2$ and all
$t>0$. Hence $\lim_{t\to 0}|\omega^t_{01}+\sum_{j\geq
2}\omega^t_{0j}|=\infty$.

For all $t\in (0,1]$,  let $(\tilde \omega^t,\tilde \lambda^t)$ be the unique equilibrium stress for $(\hat G,q^t)$ with
$\tilde \omega^t_{01}= \omega^t_{01}$. Then $(\tilde \omega^t,\tilde \lambda^t)$ is a continuous function of $t$ for $t\in (0,1]$, and $\lim_{t\to 0}(\tilde \omega^t-\omega^t ,\tilde \lambda^t-\lambda^t)=(0,0)$.
It follows that $X^T\Omega(\tilde\omega^t)X=\begin{bmatrix} \tilde A_t &  B_t\\[1mm]
 B_t^T& D_t\end{bmatrix}$ where $B_t,D_t$ converge to $B$ and $\Omega(\omega)$ respectively as $t\to 0$, and $ \tilde A_t=(a_t)$ with $a_t=\omega^t_{01}+\sum_{j\geq
2}\tilde\omega^t_{0j}$ so $\lim_{t\to 0}|a_t|=\infty$. We can now use
Lemma \ref{lem:matrixlimit} to deduce that $$\rank
\Omega(\tilde \omega^t)\geq 1+\rank \Omega(\omega)$$ when $t$ is sufficiently
close to zero.

We next consider the matrix $\Omega(\omega^t)+\Lambda(\lambda^t)$. We have
$$X^T(\Omega(\omega^t)+\Lambda(\lambda^t))X=X^T\Omega(\omega^t)X+X^T\Lambda(\lambda^t)X=\begin{bmatrix}\bar A_t & B_t\\[1mm] B_t^T & \Omega(\omega)+\Lambda(\lambda)\end{bmatrix}$$
where $\bar A_t=A_t+(\lambda_0^t)$ and $B_t=B+(\lambda_0^t,0,0,\ldots,0)$.
We will show that  $\lambda^t_0$ converges to a finite limit as $t\to \infty$.
We can then use a similar argument to that in the previous paragraph to deduce that $$\rank
(\Omega(\tilde\omega^t)+\Lambda(\tilde\lambda^t))\geq 1+\rank
(\Omega(\omega)+\Lambda(\lambda))$$
for all  $t$ sufficiently
close to zero.

Let $a^\perp$ be a unit vector
in $\span\{a,b\}$ which is orthogonal to $a$. Then Equation
(\ref{eq:v_0}) gives $\lambda^t_0\bar p(v_1)\cdot a^\perp+
\omega^t_{01}c\cdot a^\perp=0$. If $\bar p(v_1)\cdot a^\perp= 0$ then
$\bar p(v_1)$ would be a scalar multiple of $a$ and Equation
(\ref{eq:v_1}) would give $\dim \span\{a,b\}=1$. Hence $\bar
p(v_1)\cdot a^\perp\neq 0$ and we  may use Equation (\ref{eqn:wo1})
to deduce that
$$\lambda^t_0=-\frac{\omega^t_{01}c^t\cdot a^\perp}{\bar p(v_1)\cdot
a^\perp} =\frac{a\cdot \bar p(v_1)^\perp}{c^t\cdot \bar
p(v_1)^\perp}\,\frac{c^t\cdot a^\perp}{\bar p(v_1)\cdot a^\perp}=
\frac{a\cdot \bar p(v_1)^\perp}{\bar p(v_1)\cdot
a^\perp}\,\frac{c^t\cdot a^\perp}{c^t\cdot \bar p(v_1)^\perp}=
\frac{a\cdot \bar p(v_1)^\perp}{\bar p(v_1)\cdot
a^\perp}\,\frac{\hat c^t\cdot a^\perp}{\hat c^t\cdot \bar
p(v_1)^\perp}\,$$ where $\hat c^t$ is a unit vector in the direction
of $c^t$. As $t\to 0$, $q(v_0)$ approaches $p(v_1)$ on $\C$, and $\hat c^t$ will
approach $\bar p(v_1)^\perp$. Hence $\lambda^t_0$ will become
arbitrarily close to $\frac{a\cdot \bar p(v_1)^\perp\,\bar
p(v_1)^\perp\cdot a^\perp}{\bar p(v_1)\cdot a^\perp}$ (since $\bar p(v_1)^\perp$ has unit length).

As noted above, the fact that $\lambda^t_0$ converges to a finite limit as $t\to \infty$ implies that
$\rank
(\Omega(\tilde\omega^t)+\Lambda(\tilde\lambda^t))\geq 1+\rank
(\Omega(\omega)+\Lambda(\lambda))$ for all $t$ sufficiently close to $0$. This in turn implies that $(\tilde \omega^t,\tilde \lambda^t)$ is an equilibrium stress for $(\hat G,q^t)$
on $\Y$ with
$$\rank \Omega_{\CYL}(\tilde \omega^t,\tilde
\lambda^t)\geq \rank \Omega_{\CYL}(\omega,
\lambda)+3=3|V(\hat G)|-6$$ for all $t$ sufficiently close to $0$. The fact that equality holds follows since $3|V(\hat G)|-6$ is an upper bound on the rank of a stress matrix for any realisation of $\hat G$ on $\Y$. 
\end{proof}

Combining the results thus far we have the following key result.

\begin{thm}\label{thm:unitstress}
Let $(G,p)$ be a generic realisation of an $\M_{2,2}^*$-circuit on
$\Y$. Then $(G,p)$ has a maximum rank equilibrium
stress.
\end{thm}

\begin{proof}
We apply induction on $|V|$. We  give  specific infinitesimally
rigid realisations of $K_5^-$ and $H_1$ on $\Y$ which
have maximum rank equilibrium stresses in the Appendix. Theorem
\ref{thm:cylinderlaman} and Lemma \ref{lem:generic} now imply that
every generic realisation $p$ of $K_5^-$ or $H_1$ is infinitesimally
rigid on $\Y$ and has a maximum rank equilibrium stress on $\Y$. By
Theorem \ref{thm:strongrecurse} any $\M_{2,2}^*$-circuit $G$ can be
formed from $K_5^-$ or $H_1$ by $K_4^-$-extensions and generalised
vertex splits. The result now follows from Theorems \ref{lem:typea1}
and \ref{thm:vsplitmax} and Lemma \ref{lem:generic}.
\end{proof}

\section{Globally rigid frameworks}
\label{sec:globalfull}

In this section we will prove our main results, Theorems \ref{thm:full} and  \ref{thm:globrigid}.

We say that a framework $(G,p)$ on $\Y$ is
\emph{quasi-generic} if it is congruent to a generic framework on
$\Y$. The framework $(G,p)$  is said to be in \emph{standard position} on
$\Y$ if $p(v_1)=(0,1,0)$.

\begin{lem}\label{lem:zrestricted}
Let $(G, p)$ and $(G, p')$ be equivalent quasi-generic frameworks in
standard position on $\Y$. Let $p(v_i)=(x_i,y_i,z_i)$ and
$p'(v_i)=(x_i',y_i',z_i')$ for each $v_i\in V$. Suppose that $(G,p)$
has a maximum rank equilibrium stress $(\omega,\lambda)$. Then $(z_1', z_2',\dots,
z_n')=b(z_1, z_2,\dots, z_n)$ for some $b\in \mathbb{R}$.
\end{lem}

\begin{proof}
By Theorem \ref{thm:genstressPartial}, $(\omega,
\lambda')$ is an  equilibrium stress  for $(G,p')$ for some $\lambda'\in \mathbb{R}$. Let 
$$
\Omega_\CYL(\omega,\lambda)= \begin{bmatrix} \Omega + \Lambda & 0 & 0\\ 0 & \Omega + \Lambda & 0\\ 0 & 0 & \Omega \end{bmatrix} \mbox{ and } \Omega_\CYL(\omega,\lambda')=\begin{bmatrix} \Omega + \Lambda' & 0 & 0\\ 0 & \Omega + \Lambda' & 0\\ 0 & 0 & \Omega \end{bmatrix}
$$
be the stress matrices corresponding to $(\omega,
\lambda)$ and $(\omega,
\lambda')$, respectively.
Observe that the submatrix $\Omega$ is the same in both stress matrices. We saw in
Section \ref{sec:rig} that $\rank \Omega =n-2$ and that
$\{(z_1,z_2,\dots,z_n), (1,1,\dots,1)\}$ is a basis for $\coker \Omega$.
Since $(z_1',z_2',\dots,z_n')\in \coker \Omega$ we have $(z_1',
z_2',\dots, z_n')=b(z_1, z_2,\dots, z_n)+c (1, 1,\dots, 1) $ for
some $b, c \in \mathbb{R}$. The hypothesis  that $z_1=z_1'=0$, now gives $c=0$.
\end{proof}

Our next result tells us that equivalent quasi-generic frameworks which satisfy the conclusion of Lemma \ref{lem:zrestricted} are in fact congruent. Its proof requires the introduction of a new concept. We willl  delay this to Section \ref{sec:restricted} and instead show how the result can be used to prove Theorem \ref{thm:globrigid}.

\begin{lem}\label{lem:zrestricted2}
Let $G$ be a 2-connected graph with at least $n+1$ edges, and $(G, p)$ and $(G, p')$ be equivalent quasi-generic frameworks in
standard position on $\Y$.  Suppose that $p(v_i)=(x_i,y_i,z_i)$ and $p(v_i')=(x_i',y_i',z_i')$ 
for all $v_i\in V$, and that $(z_1', z_2',\dots, z_n')=b(z_1, z_2,\dots, z_n)$ for some $b\in \mathbb{R}$. Then $(G,p')$ is congruent to $(G,p)$.
\end{lem}

\begin{proof}[{\bf Proof of Theorem  \ref{thm:globrigid}}]
Suppose that $(G,p')$ is equivalent to $(G,p)$. We may assume that $(G,p)$ is in fact quasi-generic and that and $(G,p)$ and $(G,p')$ are both in standard position on $\Y$. Theorem \ref{thm:unitstress} implies that $(G,p)$ has a maximum rank equilibrium stress. Lemmas \ref{lem:zrestricted} and \ref{lem:zrestricted2} now imply that $(G,p)$ and $(G,q)$ are congruent. 
\end{proof}

We will need some further concepts and results from  matroid theory to deduce Theorem \ref{thm:full} from Theorem \ref{thm:globrigid}. A matroid is \emph{connected} if any pair of edges is contained in a common circuit.

\begin{lem}\cite[Theorem 5.4]{Nix}\label{lem:mcon}
Suppose that $G$ is a graph. Then $\M_{2,2}^*(G)$ is connected if and only if $G$ is $2$-connected and redundantly rigid on $\Y$.
\end{lem}

Let $\mathcal{M}=(E,\mathcal{I})$ be a matroid and let $C_1,C_2,\dots,C_t$ be a non-empty sequence of circuits of $\mathcal{M}$. Let $D_j=C_1\cup C_2\cup \dots \cup C_j$ for $1\leq j\leq t$ and suppose $D_t=E$. We say that $C_1,C_2,\dots,C_t$ is an \emph{ear decomposition} of $\mathcal{M}$ if for all $2\leq i\leq t$ the following properties hold:
\begin{enumerate}
\item $C_i\cap D_{i-1}\neq \emptyset$;
\item $C_i-D_{i-1}\neq \emptyset$;
\item no circuit $C_i'$ satisfying (1) and (2) has $C_i'-D_{i-1}$ properly contained in $C_i-D_{i-1}$.
\end{enumerate}

\begin{lem}[\cite{C&H}]\label{lem:ear}
A matroid is connected if and only if it has an ear decompostion.
\end{lem}

We also need a `glueing' lemma for combining globally rigid frameworks. Its proof uses  the following result.

\begin{lem}\cite[Lemma 14]{JMN}\label{lem:isom}
Let $G$ be a graph with at least five vertices and
$(G,p)$ and $(G,q)$ be congruent generic realisations of $G$ on the cylinder $\Y$.
Then  $\iota\circ p = q$ for some isometry $\iota$ of $\Y$.
\end{lem}

Note that the isometries of $\Y$ are translations along and rotations about the $z$-axis, as well as reflections in any plane containing or orthogonal to the $z$-axis. Hence Lemma \ref{lem:isom} implies that, if $(G,p)$ and $(G,q)$ are congruent generic frameworks on $\Y$ with at least five vertices satisfying $p(v_1)=q(v_1)$ and $p(v_2)=q(v_2)$ for two distinct vertices $v_1,v_2$ of $G$,  then $p=q$.

\begin{lem}\label{lem:glue}
Let $G_1$ and $G_2$ be graphs on at least five vertices and with at least two vertices in common. Let $(G,p)$ be a generic realisation of $G=G_1\cup G_2$ on $\Y$ and let $p_i=p|_{G_i}$. Suppose that $(G_i,p_i)$ is globally rigid on
$\mathcal{Y}$ for $i=1,2$. Then  $(G,p)$ is globally rigid on $\mathcal{Y}$.
\end{lem}

\begin{proof}
Choose $u,v\in V(G_1)\cap V(G_2)$ and let
$(G,q)$ be an equivalent framework to $(G,p)$ on $\Y$. By applying a suitable isometry of $\Y$ to $q$ we may assume that $p(u)=q(u)$. Since $(G_1,p_1)$ is globally rigid on $\Y$, Lemma \ref{lem:isom} tells us that $q|_{G_1}=\iota \circ p_1$ for some  
isometry $\iota$ of $\Y$.
In particular $q(v)=(\iota\circ p)(v)$. Since $(G_2,p_2)$ is globally rigid on $\Y$,  Lemma \ref{lem:isom} implies that there is a unique equivalent realisation of $G_2$ on  $\Y$ which maps $u$ to $p(u)$ and $v$ to $(\iota\circ p)(v)$. Since both  $(G_2,q|_{G_2})$  and
$(G_2,\iota\circ p_2)$ have this property,  $q|_{G_2}=\iota \circ p_2$. Hence $q=\iota\circ p$ and $(G,p)$ is congruent to $(G,q)$.
\end{proof}

\begin{proof}[{\bf Proof of Theorem  \ref{thm:full}}]
Necessity follows from \cite{JMN}. We prove sufficiency by induction on $|E|$. Since $G$ is 2-connected and redundantly rigid, $\M_{2,2}^*(G)$ is connected by Lemma \ref{lem:mcon}, so has an ear decomposition by Lemma \ref{lem:ear}. Let $H_1,H_2,\ldots,H_t$ be the $\M_{2,2}^*$-circuits induced by an ear decomposition of $\M_{2,2}^*(G)$ and $p_i=p|_{H_i}$. Then $(H_i,p_i)$ is globally rigid on $\Y$ for all $1\leq i\leq t$ by Theorem \ref{thm:globrigid}. Since $|V(H_i)|\geq 5$ and $|(\cup_{j=1}^i V(H_j))\cap V(H_{i+1})|\geq 2$ for all $1\leq i<t$, we may deduce that $(G,p)$ is globally rigid on $\Y$ by repeated applications of Lemma \ref{lem:glue}.
\end{proof}

Theorem \ref{thm:full} implies that global rigidity on the cylinder is a generic property. That is, if $G$ is a graph and $p$ is generic on $\Y$, then $(G,p)$ is globally rigid on $\Y$ if and only if $(G,q)$ is globally rigid on $\Y$ for all generic $q$.
It also gives a polynomial time algorithm for checking generic global rigidity on the cylinder since 2-connectivity \cite{Tar} and redundant rigidity \cite{B&J,L&S} can both be checked efficiently.

\section{Cylindrical frameworks with one free vertex}
\label{sec:vfree}

In this section  we characterise rigidity for a  generic framework on the
cylinder when one designated vertex is allowed to move off the
cylinder. This is used to prove Lemma \ref{lem:expand}.

Let $(G,p)$ be a generic realisation of an $\M_{2,2}^*$-circuit on  $\Y$,  and $(\omega,\lambda)$ be a non-zero stress for
$(G,p)$. If $\lambda_i=0$ for some vertex $v_i$ then
$(\omega,\lambda|_{V-v_i})$ would be in the cokernel of the matrix
$R_{v_i}(G,p)$ obtained from $R_{\CYL}(G,p)$ by deleting the row
indexed by $v_i$. Hence the rows of $R_{v_i}(G,p)$ would be
dependent. Thus it will suffice to show that this cannot occur i.e.
$\rank R_{v}(G,p)=3|V|-2$ for all $v\in V$.

We can view $R_{v}(G,p)$ as a rigidity matrix for the framework
$(G,p)$ when $v$ is free to move arbitrarily in $\bR^3$ and all
other vertices are constrained to stay on $\Y$. We will refer to
such (infinitesimal) motions as {\em $v$-free (infinitesimal)
motions} and say that $(G,p)$ is {\em $v$-free (infinitesimally)
rigid} if its only {$v$-free (infinitesimal) motions} are induced by
the continuous isometries of $\Y$. Thus $(G,p)$ is $v$-free infinitesimally rigid if
and only if $\rank R_{v}(G,p)=3|V|-2$  and Lemma
\ref{lem:expand} will follow immediately from our next result.

\begin{thm}\label{thm:vfree}
Let $(G,p)$ be a generic realisation of an $\M_{2,2}^*$-circuit on $\Y$ and $v\in
V$. Then $(G,p)$ is $v$-free infinitesimally rigid.
\end{thm}

Note that  `$v$-free infinitesimal rigidity'  is a generic property
on $\Y$ since it is characterised by the rank of $R_{v}(G,p)$. We will say that {\em $G$ is $v$-free rigid on $\Y$} if
some, or equivalently every, generic realisation of $G$ on $\Y$ is
$v$-free infinitesimally rigid. Given graphs $H$ and $G$, we say
that $H$ is a {\it $0$-extension} of $G$ if $H$ can be obtained from $G$
by adding a vertex of degree two. We will need the following result
in the proof of Theorem \ref{thm:vfree}.

\begin{lem}\label{lem:vfree}
Let $G=(V,E)$ be
a graph and $v\in V$.\\
(a) Suppose that $G$ is a $0$-extension of a graph $H$ obtained by adding a new vertex $v_0$ and new edges $v_0v_1$ and $v_0v_2$, $v\in V(H)$
and $H$ is $v$-free rigid on $\Y$. Then $G$ is $v$-free rigid on
$\Y$. More precisely, if $(G,p)$ is any framework on $\Y$ such that 
$(H,p|_H)$ is $v$-free infinitesimal rigid and $p(v_0),p(v_1),p(v_2)$ are not collinear, then $(G,p)$ is $v$-free infinitesimal rigid.
\\
(b) Suppose that $G$ is a $1$-extension of a  graph $H$, $v\in V(H)$
and $H$ is $v$-free rigid on $\Y$. Then $G$ is
$v$-free rigid on $\Y$.\\
(c) Suppose that $G=H_1\cup H_2$, $v\in V(H_1)$, $H_1$ is $v$-free
rigid on $\Y$, $H_2$ is rigid on $\Y$ and $H_1\cap H_2 \neq
\emptyset$. Then $G$ is $v$-free rigid on $\Y$.\\
(d) Suppose that $H$ is a subgraph of $G-v$ with at least four
vertices, $H$ is rigid on $\Y$, and $G/H$
is $v$-free rigid on $\Y$. Then $G$ is $v$-free rigid on $\Y$.\\

\end{lem}

\begin{proof}
(a) This follows since the definition of $R_v(G,p)$ implies that $$\rank R_v(G,p)=\rank
R_v(H,p|_H)+3=3|V|-2.$$

(b) Let $V=\{v_0,v_1,\dots,v_n\}$, let $(G,p)$ be generic on $\Y$
and suppose $G$ is formed from $H$ by deleting the edge $v_1v_2$ and
then adding the vertex $v_0$ and edges $v_0v_i$ for $1\leq i\leq 3$.
By symmetry between $v_1$ and $v_2$, we may assume that $v_2\neq v$.

Suppose that $(G,p)$ is not $v$-free infinitesimally rigid on $\Y$. Then it follows that no framework $(G,q)$ on $\Y$ is $v$-free infinitesimally rigid. Let $\P$ be the tangent plane to $\Y$ at $p(v_2)$ and let $a$ and $b$ be orthogonal unit vectors in $\P$ such that $b$ is orthogonal to $p(v_2)-p(v_1)$. Consider a sequence of frameworks $(G, q^k)$ on $\Y$, in which $q^k|_{V-v_0}=p|_{V-v_0}$  and $q^k(v_0)$ tends to $p(v_2)$ in the direction of $a$ as $k\to \infty$. More precisely, the normalised vector $(p(v_2)-q^k(v_0))/\|p(v_2)-q^k(v_0)\|$ converges to $a$, as $k\to \infty$. Let $q=\lim_{k\to \infty} q^k$.

Since $(G, q^k)$ is not $v$-free infinitesimally rigid, $(G, q^k)$ has a unit norm infinitesimal motion $m^k$ which has $m^k(v_2)=0$ and hence is not an infinitesimal isometry of the cylinder. By the Bolzano-Weierstrass theorem there is a subsequence of the sequence $m^k$ which converges to a vector, $m=(m_0, m_1,\dots,m_n)$ say. Relabeling if necessary, we may assume this holds for the original sequence $m^k$.  Then $m$ is a $v$-free infinitesimal motion of the framework  $(G,q)$, for which $q(v_0)=p(v_2)$ and $q(v_i)=p(v_i)$ for all $1\leq i\leq n$.

We claim that $m$  is a $v$-free infinitesimal motion of $(G+v_1v_2,q)$.  To see this, it will suffice to show that  $m_2 -m_1$  is orthogonal to $p(v_2)-p(v_1)$.  Since $m$ is a $v$-free infinitesimal motion of  $(G,q)$,  $m_0,m_2\in \P$ and, since $v_2v_0\in E(G)$ and    $(p(v_2)-q^k(v_0))/\|p(v_2)-q^k(v_0)\|$ converges to $a$,  $m_2-m_0$ is orthogonal to $a$. This tells us that $m_2-m_0$ is a scalar multiple of $b$ and hence that $m_2-m_0$ is orthogonal to $p(v_2)-p(v_1)$.
The fact that $m$ is a $v$-free infinitesimal motion of  $(G,q)$ and $v_1v_0\in E(G)$, also implies that $m_1-m_0$ is orthogonal to $q(v_0)-q(v_1)=p(v_2)-p(v_1)$. Taking differences, we may deduce that $m_2 -m_1=(m_2-m_0)-(m_1-m_0)$  is orthogonal to $p(v_2)-p(v_1)$, as desired.

Since the vectors $m^k$ have unit norm and satisfy 
$m^k(v_2)=0$, the vector $m$ will also have unit norm and satisfy 
$m(v_2)=0$. Hence $(G+v_1v_2,q)$ is not
$v$-free infinitesimally rigid. This contradicts (a) since
$(H,q|_H)$ is $v$-free infinitesimally rigid, $G+v_1v_2$ can be
obtained from $H$ by adding the edges $v_0v_1,v_0v_3$ (and
$v_0v_2$), and $q(v_0)=p(v_2), q(v_1)=p(v_1), q(v_3)=p(v_3)$
are not collinear.

(c) Let $(G,p)$ be a generic realisation of $G$ on $\Y$ and $m$ be
an infinitesimal motion  of $(G,p)$.
By applying an infinitesimal isometry of $\Y$ to $(G,p)$ we may
assume that $m(v_0)=0$ for some $v_0\in V(H_1-v)$. The fact that
$H_1$ is $v$-free rigid on $\Y$ now implies that the restriction of
$m$ to $(H_1,p|_{H_1})$ is zero. In particular, $m(u)=0$ for all
$u\in V(H_1)\cap V(H_2)$.  Since $H_2$ is rigid on $\Y$ and $H_1\cap
H_2\neq \emptyset$,  the restriction of $m$ to $(H_2,p|_{H_2})$ is
zero. Hence $m=0$ and $G$ is $v$-free rigid on $\Y$.

(d) Let $V=\{v_1,v_2,\ldots,v_n\}$, $V(H)=\{v_1,v_2,\ldots,v_r\}$
and $(G,p)$ be a generic framework on $\Y$.
By reordering its rows, we can write $R_v(G,p)$ in the form
\[\begin{pmatrix} R_{\CYL}(H,p|_H) & 0 \\ M_1(p) & M_2(p) \end{pmatrix} \]
where $M_2(p)$ is a matrix with $3(n-r)$ columns.

Suppose, for a contradiction, that $G$ is not $v$-free rigid on
$\Y$. Then there exists a vector $m\in \ker R_v(G,p)$ which is not
an infinitesimal isometry of $\Y$. Since $(H,p|_H)$ is rigid on $\Y$
we may suppose that $m=(0,\dots,0,m_{r+1},\dots,m_n)$. Consider the
realisation $(G,p')$ where $p'=(p(v_r), p(v_r), \dots,p(v_r),
p(v_{r+1}), \dots, p(v_n))$. Let $v_r^*$ be the vertex of $G/H$
which corresponds to $H$ and define the realisation $(G/H,p^*)$ by
setting $p^*(v_r^*)=p(v_r)$ and $p^*(v_i)=p(v_i)$ for all $r<i\leq
n$.  Since $p^*$ is generic, $(G/H,p^*)$ is $v$-free rigid on $\Y$
by assumption.

Since the nonzero vector $(m_{r+1},\dots,m_n)\in \ker M_2(p)$,
$\rank M_2(p) < 3(n-r)$. Since $p$ is generic, we also have $\rank
M_2(p') \leq\rank M_2(p)< 3(n-r)$ and hence there exists a nonzero vector $m'\in
\ker M_2(p')$. Therefore we have
\[ \begin{pmatrix} R_v(G/H,p^*)\end{pmatrix} \begin{pmatrix}0 \\ m' \end{pmatrix}= \begin{pmatrix} \star & M_2(p')\\ \bar p(v_r)& 0  \end{pmatrix}\begin{pmatrix} 0 \\ m' \end{pmatrix}=0. \]
Thus  $(0,m')$ is a nonzero infinitesimal motion of $(G/H,p^*)$
which fixes $v_r^*$. This contradicts the $v$-free rigidity of
$(G/H,p^*)$.
\end{proof}

\begin{proof}[\bf Proof of Theorem \ref{thm:vfree} (and Lemma \ref{lem:expand})]
We use a similar inductive proof technique to that of Theorem
\ref{strongadmiss}. The base graphs in our induction are the graphs
$K_5^-,H_1,H_2$. We give an infinitesimally rigid realisation
$(G,p)$ of $G$ on $\Y$ with a {\em nowhere zero} stress for each $G\in
\{K_5^-,H_1,H_2\}$ in the Appendix. This implies that the rank of $
R(G,p)$ is $3|V|-2$ and that its rank will remain unchanged if we delete any of its rows. In particular, we have $\rank
R_v(G,p)=3|V|-2$  and hence $G$ is
$v$-free rigid on $\Y$. Thus we may assume that $G\not\in
\{K_5^-,H_1,H_2\}$.

Suppose that $G$ has an admissible node $x$. If $x\neq v$ then the
graph obtained from $G$ by performing an admissible $1$-reduction at
$x$ is $v$-free rigid on $\Y$ by induction and hence $G$ is $v$-free
rigid on $\Y$ by Lemma \ref{lem:vfree}(b). On the other hand, if
$x=v$ then $v$ has degree three. Since $G$ is an $\M_{2,2}^*$-circuit, $G-v$ is
rigid on $\Y$. This implies that $G$ is $v$-free rigid on $\Y$ since $v$
has three neighbours in $G-v$.

Hence we may assume that $G$ has no admissible nodes. Theorem
\ref{thm:nearly3con} now implies that $G$ has a nontrivial
2-separation or 3-edge-separation.  It follows that $G$ has at least
two distinct atoms. Thus we may choose a  nontrivial separation
$(F_1,F_2)$ of $G$ such that $F_1$ is an atom of $G$ and $v\in
V(F_2)$.  Consider the following cases.\\[1mm]
{\bf Case 1}: $(F_1,F_2)$ is a nontrivial 3-edge-separation.

Let $G_i$ be obtained from $G$ by contracting $F_{3-i}$ to a single
vertex $z_i$ for each $i\in \{1,2\}$. Then $G=G_1*_3 G_2$ so
$G_1,G_2$ are $\M_{2,2}^*$-circuits on $\Y$ by Theorem
\ref{thm:circuit_dec}. Since $G_1$ is an $\M_{2,2}^*$-circuit on
$\Y$ and $d(z_1)=3$, $G_1-z_1$ is rigid on $\Y$, and since $G_2$ is an
$\M_{2,2}^*$-circuit, $G_2$ is $v$-free rigid on $\Y$ by
induction. Lemma \ref{lem:vfree}(d) now implies that $G$ is $v$-free
rigid on $\Y$.

\medskip
\noindent {\bf Case 2}: $(F_1,F_2)$ is a nontrivial
2-vertex-separation, $V(F_1)\cap V(F_2)=\{x,y\}$ and $E(F_1)\cap
E(F_2)=\{xy\}$.

Let $G_i$ be obtained from $F_i$ by adding two new vertices
$\{w,z\}$ and five new edges $\{wx,wy,wz,xz,yz\}$. Then $G=G_1*_2
G_2$ so $G_1$ and $G_2$ are $\M_{2,2}^*$-circuits by Theorem
\ref{thm:circuit_dec}. Since $F_1$ is an atom, $G_1$ has no
nontrivial 2-vertex-separation or 3-edge-separation. If $G_1\not\in
\{K_5^-,H_1,H_2\}$, then $G_1$ would have an admissible node $u\neq v$ by
Theorem \ref{thm:nearly3con}. It is straightforward to check that $u$ would be an admissible node in $G$. Since $G$
has no admissible nodes, we may deduce that  $G_1\in
\{K_5^-,H_1,H_2\}$. We have $G_1\neq K_5^-$ since $G_1$ is
not 3-connected and   $G_1 \neq H_1$ since the 2-separation
$(F_1,F_2)$ is nontrivial. Hence $G_1 = H_2$ and $G$ is as shown in
Figure \ref{fig:case2_app}.

Put $\{b,c\}= \{x,y\}$ with $b\neq v$.
Let $(G_2,p_2)$ be a generic realisation of $G_2$ on
$\Y$.
Since
$G_2$ is an $\M_{2,2}^*$-circuit, $G_2$ is  $v$-free rigid by
induction. Since $R_v(G_2,p)$ has $3|V(G_2)|-2$ rows, $(G_2-wz,p_2)$ will have a unique nonzero $v$-free
infinitesimal motion which fixes $b$ (up to scalar multiplication).
Since $w,z$ have degree two in $G_2-wz$, $(F_2,p_2|_{F_2})$ also has a
unique nonzero $v$-free infinitesimal motion $m_2$ which fixes $b$.
The
fact that $G_2$ is $v$-free rigid implies that $m_2(c)\neq 0$.
Since $G_1$ is an $\M_{2,2}^*$-circuit and $F_1=G_1-\{w,z\}$, we can choose a realisation $(F_1,p_1)$ which is
infinitesimally rigid on $\Y$ and has $p_1(x)=p_2(x)$ and $p_1(y)=p_2(y)$.
Let $(G,p)$ be the realisation of $G$ on $\Y$ with $p|_{F_i}=p_i$. 
We will show that $(G,p)$ is infinitesimally $v$-free rigid.  
We assume, for a contradiction, that $(G,p)$ has a nonzero $v$-free infinitesimal motion $m$  which keeps $b$ fixed.

Suppose that $v\not\in\{x,y\}$. Then, since $(F_1,p_1)$ is
infinitesimally rigid on $\Y$, $m|_{F_1}=0$. This contradicts the fact that we must have $m(c)\neq 0$, since $m|_{F_2}$ will be a nonzero scalar multiple of the motion $m_2$ defined in the previous paragraph.
Hence $v=c$.

The fact that $(G_1,p_1)$ is
$v$-free rigid implies that $(F_1,p_1)$ has a unique nonzero
$v$-free infinitesimal motion $m_1$ which keeps $b$ fixed (up to
scalar multiplication). Thus $m(v)=\alpha m_1(v)=\beta m_2(v)$
for some $\alpha,\beta\in \bR$. We will show that, as we
vary the positions of the vertices of $F_1-\{x,y\}$ in $(G,p)$, the direction of $m_1(v)$ changes. Since the direction of $m_2(v)$ remains fixed this will imply that we can choose $p_1$ such that $(G,p)$ is infinitesimally $v$-free rigid.

When $v=y$ this follows since the $K_4$ subgraph of $F_1$ attached to $x$ is rigid on $\Y$ and hence $m_1|_{F_1-y}=0$. Thus $m_1(y)$ is normal to the plane through $p_1(x),p_1(v_4),p_1(y)$. Clearly the direction of $m_1(y)$ will change as we vary $p_1(v_4)$. Hence we may suppose that
$v=x$. In this case we first apply an isometry of $\Y$ to $(G,p)$ so that $p(y)=(0,1,0)$ and then put $p(x)=(x_1,y_1,z_1)$, $p(v_2)=(1,0,-1)$, $p(v_3)=(-1,0,-1/3)$ and $p(v_4)=(1/\sqrt{2},1/\sqrt{2},z_4)$.  We may scale $m_1$ so that $m_1(x)=(1,u_5,u_6)$ and then solve $R_v(F_1,p_1)m_1=0$ to deduce that $ u_5=-\frac{(1+\sqrt{2})(\sqrt{2}x_1y_1z_4^2-x_1z_1z_4+z_1^2-z_4z_1)}{z_4(2y_1z_4+z_1-2z_4+\sqrt{2}y_1z_4-\sqrt{2}z_4)y_1}. $
Hence the direction of $m_1(x)$ will change as we vary $z_4$.

\begin{center}
\begin{figure}
\begin{tikzpicture}[scale=0.5]

%

 \draw[black,thick]
(0,.5) -- (0,2.5);

\draw (1,1.5) circle (57pt);

\filldraw (0,2.5) circle (3pt)node[anchor=south]{$x$};
\filldraw (0,.5) circle (3pt)node[anchor=north]{$y$};
\filldraw (-3,3) circle (3pt)node[anchor=south]{$v_2$};
\filldraw (-4,0) circle (3pt)node[anchor=north]{$v_3$};
\filldraw (-2,0) circle (3pt)node[anchor=north]{$v_4$};

\draw[black,thick] (0,2.5) -- (0,.5);

 \draw[black,thick]
(0,.5) -- (-2,0) -- (-4,0) -- (-3,3) -- (0,2.5) -- (-4,0);

\draw[black,thick] (-3,3) -- (-2,0) -- (0,2.5);


  \node [rectangle, draw=white, fill=white] (b) at (-1,-2) {$G$};

\filldraw (-5,1) circle (0pt)node[anchor=south]{$F_1$};

\filldraw (1.5,1) circle (0pt)node[anchor=south]{$F_2$};

\end{tikzpicture}
\caption{The graph $G$ in the subcase $G_1=H_2$ of Case 2.} \label{fig:case2_app}
\end{figure}
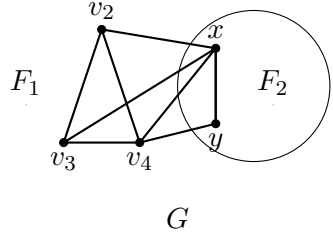
\end{center}

\medskip
\noindent {\bf Case 3}: $(F_1,F_2)$ is a nontrivial
2-vertex-separation, $V(F_1)\cap V(F_2)=\{x,y\}$ and $xy\not\in E$.

We have $$|E(F_1)|+|E(F_2)|=|E(G)|=2|V(G)|-1=2|V(F_1)|+2|V(F_2)|-5$$
and $|E(F_i)|\leq 2|V(F_i)|-2$ for each $i=1,2$ so $2|V(F_i)|-3\leq
|E(F_i)|\leq 2|V(F_i)|-2$. Consider the following subcases.

\smallskip \noindent {\bf Subcase 3.1}: $|E(F_1)|= 2|V(F_1)|-3$.
Let $G_1$ be obtained from $F_1$ by adding two new vertices
$\{w,z\}$ and six new edges $\{wx,wy,wz,xy,xz,yz\}$, and
$G_2=F_2+xy$. Then $G=G_1*_1 G_2$ so $G_1$ and $G_2$ are $\M_{2,2}^*$-circuits by
Theorem \ref{thm:circuit_dec}. Since $F_1$ is an atom, $G_1$ has no
nontrivial 2-vertex-separation or 3-edge-separation. If $G_1\not\in
\{K_5^-,H_1,H_2\}$, then $G_1$ would have an admissible node $u\neq v$ by
Theorem \ref{thm:nearly3con}. It is straightforward to check that $u$ would be an admissible node in $G$. Since $G$ has no admissible nodes we may deduce that
$G_1\in
\{K_5^-,H_1,H_2\}$.
We have $G_1\neq K_5^-$ since $G_1$ is
not 3-connected. Hence $G_1 \in  \{H_1,H_2\}$.

Suppose $G_1=H_1$. Then $F_1=K_4^-$ and $G$ is a $K_4^-$-extension
of $G_2$. Let $(G_2,p_2)$ be a generic realisation of $G_2$ on $\Y$.
Then $(G_2,p_2)$ is $v$-free infinitesimally rigid by induction. Let
$(G+xy,p)$ be a realisation of $G+xy$ on $\Y$ such that
$p|_{F_2}=p_2$ and such that the vertices of $p|_{F_1}$ are coplanar
and are in general position on this plane. Then $(G+xy,p)$ is
$v$-free infinitesimally rigid since it can be obtained from
$(G_2,p_2)$ by two $0$-extensions  and an edge addition. In addition
the rows of $R_v(G+xy)$ indexed by $E(F_1)+xy$ are minimally
linearly dependent since the vertices of $p|_{F_1}$ are coplanar and
are in general position on this plane. Hence we may delete the edge
$xy$ from $(G+xy,p)$ without destroying its $v$-free infinitesimal
rigidity. Thus we may assume that $G_1\neq H_1$.

Suppose $G_1=H_2$. Then $G$ is as shown in Figure
\ref{fig:K_5^-_app}(i). If $v\neq x$ then we may apply Lemma
\ref{lem:vfree}(d)  with $H=F_1-y$ to deduce that $G$ is $v$-free
rigid on $\Y$. Hence $v=x$. Let $(G_2,p_2)$ be a generic realisation
of $G_2$ on $\Y$.  Since $(G_2,p|_{G_2})$ is $v$-free rigid by
induction, $(F_2,p|_{F_2})$ has a unique nonzero infinitesimal
motion $m_2$ which keeps $y$ fixed (up to scalar multiplication),
and $ m_2(x)\neq 0$.

Since $G_1$ is $v$-free rigid, we can choose a realisation $(F_1,p_1)$ such that
$p_1$ agrees with $p_2$ at $x,y$ and such that the space, $Z_1$, of infinitesimal motions of $(F_1,p_1)$ which keep $y$
fixed, is $2$-dimensional.
Let $(G,p)$ be the framework with $p|_{F_i}=p_i$.
We will show that we can choose $p_1$ such that the infinitesimal velocities $m_1(x)$ for $m_1\in Z_1$ span a 2-dimensional subspace $\P$ of $\bR^3$, and that, as we vary the positions of the
vertices of $F_1-\{x,y\}$, the normals to $\P$ span $\bR^3$. This will imply that we
can choose $p_1$ so that no nonzero infinitesimal velocity $m_1(x)$ for $m_1\in Z_1$,  has the same direction as $m_2(x)$, and hence that $(G,p)$ is infinitesimally rigid.

We again apply an isometry of $\Y$ to $(G,p)$ so that $p(y)=(0,1,0)$, and choose $p(x)$ and $p(v_i)$, $2\leq i\leq 4$, to be as defined
in the last paragraph of Case 2.
Then $(F_1,p_1)$ has two linearly independent infinitesimal motions $m_1,\tilde m_1$ with $m_1(y)=(0,0,0)=\tilde m_1(y)$, $m_1(x)=(1,0,-\frac{-\sqrt{2}x_1z_4+2x_1y_1z_4+\sqrt{2}z_1-\sqrt{2}z_4}{2z_4y_1(z_1-z_4)})$
and $\tilde m_1(x)=(0,1,-\frac{-\sqrt{2}+2y_1}{2(z_1-z_4)})$.
The normal vector to $\P$ is $n_1=m_1(x) \times \tilde m_1(x)=(\frac{-\sqrt{2}x_1z_4+2x_1y_1z_4+\sqrt{2}z_1-\sqrt{2}z_4}{2z_4y_1(z_1-z_4)} , \frac{-\sqrt{2}+2y_1}{2(z_1-z_4)},1)$. 
We can obtain three linearly independent normal vectors by evaluating
$n_1$ at each $z_4\in \{1,2,-1\}$. 
Hence the normals to $\P$ span $\bR^3$ as required.

\smallskip \noindent {\bf Subcase 3.2}: $|E(F_1)|= 2|V(F_1)|-2$.
Let $G_2$ be obtained from $F_2$ by adding two new vertices
$\{w,z\}$ and six new edges $\{wx,wy,wz,xy,xz,yz\}$, and
$G_1=F_1+xy$. Then $G=G_1*_1G_2$ so $G_1$ and $G_2$ are $\M_{2,2}^*$-circuits by
Theorem \ref{thm:circuit_dec}.
Let $(G_2,p_2)$ be a generic realisation of $G_2$ on $\Y$. Since $G_1$ is an $\M_{2,2}^*$-circuit and $F_1=G_1-xy$, we can choose a realisation $(G_1,p_1)$ such that
$(F_1,p_1)$ is infinitesimally rigid on $\Y$ and $p_1$ agrees with $p_2$ on $\{x,y\}$. Let $(G,p)$ be the framework with $p|_{F_i}=p_i$.

Suppose $v\not\in \{x,y\}$. Since $v\in V(F_2)$ we have $v\not\in
V(F_1)$. Let $m$ be an infinitesimal motion of $(G,p)$ which keeps
$y$ fixed. Since $(F_1,p_1)$ is infinitesimally rigid on $\Y$, $m|_{F_1}=0$ and, in particular, $m(x)=0$. It
now follows that $m|_{F_2}$ can be extended to an infinitesimal
motion $m_2$ of $(G_2,p_2)$  by
putting $m_2(w)=m_2(z)=0$. Since $(G_2,p_2)$ is infinitesimally $v$-free rigid by
induction and $m_2(y)=0$, we have $m_2=0$. Hence $m=0$ and $(G,p)$
is infinitesimally $v$-free rigid. Thus we may assume that $v\in \{x,y\}$.

Let $\{x,y\}=\{v,b\}$. Since $G_2$ is $v$-free rigid by induction, the space $Z_2$ of infinitesimal
motions of  $(F_2,p|_{F_2})$ which keep $b$ fixed
 is 2-dimensional. Similarly, since $G_1$ is
$v$-free rigid by induction, the space $Z_1$ of infinitesimal
motions of  $(F_1,p|_{F_1})$ which keep $b$ fixed
 is 1-dimensional. Our
strategy to show that $(G,p)$ is $v$-free infinitesimally rigid is to  adjust the positions of the vertices
of $F_1-\{x,y\}$ so that, for all nonzero $m_1\in Z_1$ and $m_2\in
Z_2$, $m_1(v)\neq m_2(v)$. We will accomplish this by showing that the vectors $m_1(v)$ span $\bR^3$ as $p_1|_{F_1-\{x,y\}}$ varies.

Suppose $v$ has degree three in $G_1$ and let $b,c,d$ be the
neighbours of $v$ in $G_1$. Since $G_1$ is an $\M_{2,2}^*$-circuit,
$G_1-v$ is rigid on $\Y$ and hence $m_1(u)=0$ for all $u\in
V(G_1)-v$ and all $m_1\in Z_1$. It follows that $m_1(v)$ is normal to the plane  through $p(v)$, $p(c)$ and
$p(d)$ for all $m_1\in Z_1$. This implies that we can choose $p(c)$
and $p(d)$ to ensure that $m_1(v)\neq m_2(v)$ for all $m_1\in Z_1$
with $m_1(v)\neq 0$ and all $m_2\in Z_2$. Hence we may assume that $v$ has degree at
least four in $G_1$.

Since $F_1$ is an atom, $G_1$ has no nontrivial 2-vertex-separation
or 3-edge-separation. If $G_1\not\in \{K_5^-,H_1,H_2\}$, then $G_1$
has two admissible nodes by Theorem \ref{thm:nearly3con}. Since $v$
has degree at least four in $G_1$, $G_1$ has at least one admissible
node which is distinct from $v,b$, and this node is an admissible
node of $G$. This contradicts the fact that $G$ has no admissible nodes and hence we may assume that $G_1\in \{K_5^-,H_1,H_2\}$.

Suppose $G_1=K_5^-$. Then $G$ has the structure of one of the graphs
shown in Figure \ref{fig:K_5^-_app}(ii) and (iii).
 If case (ii) occurs then we may assume by symmetry that $v=x$. On
 the other hand, if case (iii) occurs then we must also have $v=x$ since $v$ has degree
 at least four in $G_1$. In both cases we consider the realisation $(G,p)$ where $p|_{F_1}$ is as defined in the last paragraph of Case 2, and determine the unique vector $ m_1(v)=(1,u_5,u_6)$, for $m_1\in Z_1$. We then show that the vectors $m_1(v)$ span $\bR^3$ as the positions of the
vertices of $F_1-\{x,y\}$ vary. For simplicity in the expressions we also specify that $x_1=0,y_1=-1$ and $z_1=1$.

When case (ii) occurs we obtain $u_5=\frac{30z_4^2-5z_4\sqrt{2}-12z_4+3\sqrt{2}+2}{-z_4\sqrt{2}+6z_4^2-3z_4}$ and $u_6=\frac{-3z_4\sqrt{2}+18z_4^2-3z_4+6\sqrt{2}+4}{-z_4\sqrt{2}+6z_4^2-3z_4}$ and when case (iii) occurs we obtain
$u_5=\frac{-\sqrt{2}z_4+6z_4^2-6z_4+\sqrt{2}}{z_4(-\sqrt{2}+6z_4+2)}$ and $u_6=-\frac{2(3z_4-\sqrt{2}}{z_4(-\sqrt{2}+6z_4+2)}$.
In both cases, we can obtain three linearly independent infinitesimal velocities $m_1(v)$ by choosing $z_4\in\{1,-1,2\}$ in the above expressions for $(1,u_5,u_6)$.

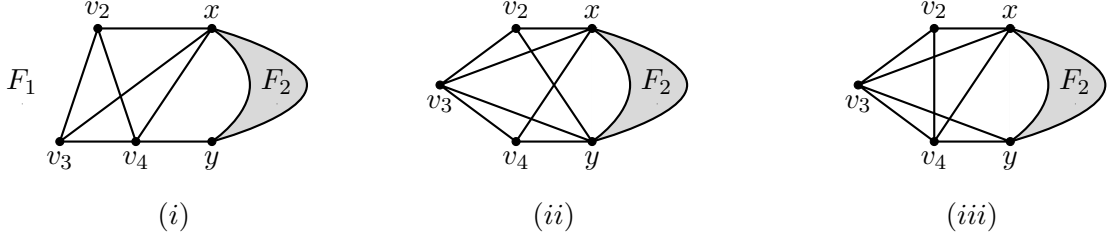
\begin{figure}
\begin{tikzpicture}[scale=0.5]

\draw[draw=gray!30!white,fill=gray!30!white] 
    plot[smooth, tension=1] coordinates{(-10,0) (-7.5,1.5) (-10,3)}  -- 
    plot[smooth, tension=1] coordinates{(-10,0) (-7.5,1.5) (-10,3)};

\draw[draw=white,fill=white] 
    plot[smooth, tension=1] coordinates{(-10,0) (-9,1.5) (-10,3)};

 \draw[white,thick]
(-10,0) -- (-10,3);

\filldraw (-10,3) circle (3pt)node[anchor=south]{$x$};
\filldraw (-10,0) circle (3pt)node[anchor=north]{$y$};
\filldraw (-13,3) circle (3pt)node[anchor=south]{$v_2$};
\filldraw (-14,0) circle (3pt)node[anchor=north]{$v_3$};
\filldraw (-12,0) circle (3pt)node[anchor=north]{$v_4$};

 \draw[black,thick]
(-10,0) -- (-12,0) -- (-14,0) -- (-13,3) -- (-10,3) -- (-14,0);

\draw[black,thick] (-13,3) -- (-12,0) -- (-10,3);

\draw[thick] plot[smooth, tension=1] coordinates{(-10,0) (-9,1.5)
(-10,3)}; \draw[thick] plot[smooth, tension=1] coordinates{(-10,0)
(-7.5,1.5) (-10,3)};

  \node [rectangle, draw=white, fill=white] (b) at (-11,-2) {$G$};

\filldraw (-8.3,1) circle (0pt)node[anchor=south]{$F_2$};
\filldraw(-15,1) circle (0pt)node[anchor=south]{$F_1$};

  \node [rectangle, draw=white, fill=white] (b) at (-11,-2) {$(i)$};


\draw[draw=gray!30!white,fill=gray!30!white] 
    plot[smooth, tension=1] coordinates{(0,0) (2.5,1.5) (0,3)}  -- 
    plot[smooth, tension=1] coordinates{(0,0) (2.5,1.5) (0,3)};

\draw[draw=white,fill=white] 
    plot[smooth, tension=1] coordinates{(0,0) (1,1.5) (0,3)};

 \draw[white,thick]
(0,0) -- (0,3);

\filldraw (0,3) circle (3pt)node[anchor=south]{$x$};
\filldraw (0,0) circle (3pt)node[anchor=north]{$y$};
\filldraw (-2,3) circle (3pt)node[anchor=south]{$v_2$};
\filldraw (-4,1.5) circle (3pt)node[anchor=north]{$v_3$};
\filldraw (-2,0) circle (3pt)node[anchor=north]{$v_4$};

 \draw[black,thick]
(0,0) -- (-2,0) -- (-4,1.5) -- (-2,3) -- (0,3) -- (-4,1.5) -- (0,0);

\draw[black,thick] (-2,3) -- (0,0);

\draw[black,thick] (-2,0) -- (0,3);

\draw[thick] plot[smooth, tension=1] coordinates{(0,0) (1,1.5)
(0,3)}; \draw[thick] plot[smooth, tension=1] coordinates{(0,0)
(2.5,1.5) (0,3)};

  \node [rectangle, draw=white, fill=white] (b) at (-1,-2) {$(ii)$};

\filldraw (1.7,1) circle (0pt)node[anchor=south]{$F_2$};


\draw[draw=gray!30!white,fill=gray!30!white] 
    plot[smooth, tension=1] coordinates{(11,0) (13.5,1.5) (11,3)}  -- 
    plot[smooth, tension=1] coordinates{(11,0) (13.5,1.5) (11,3)};

\draw[draw=white,fill=white] 
    plot[smooth, tension=1] coordinates{(11,0) (12,1.5) (11,3)};

 \draw[white,thick]
(11,0) -- (11,3);

\filldraw (11,3) circle (3pt)node[anchor=south]{$x$};
\filldraw (11,0) circle (3pt)node[anchor=north]{$y$};
\filldraw (9,3) circle (3pt)node[anchor=south]{$v_2$};
\filldraw (7,1.5) circle (3pt)node[anchor=north]{$v_3$};
\filldraw (9,0) circle (3pt)node[anchor=north]{$v_4$};

 \draw[black,thick]
(11,0) -- (9,0) -- (7,1.5) -- (9,3) -- (11,3) -- (7,1.5) --
(11,0);

\draw[black,thick] (9,3) -- (9,0) -- (11,3);

\draw[thick] plot[smooth, tension=1] coordinates{(11,0) (12,1.5)
(11,3)}; \draw[thick] plot[smooth, tension=1] coordinates{(11,0)
(13.5,1.5) (11,3)};

  \node [rectangle, draw=white, fill=white] (b) at (10,-2) {$(iii)$};

\filldraw (12.7,1) circle (0pt)node[anchor=south]{$F_2$};
\end{tikzpicture}
\caption{The graph $G$ in the subcases $G_1=H_2$ of Subcase 3.1 and $G_1=K_5^-$ of Subcase 3.2.} \label{fig:K_5^-_app}
\end{figure}

\begin{figure}
\begin{tikzpicture}[scale=0.5]

\draw[draw=gray!30!white,fill=gray!30!white] 
    plot[smooth, tension=1] coordinates{(0,0) (2.5,1.5) (0,3)}  -- 
    plot[smooth, tension=1] coordinates{(0,0) (2.5,1.5) (0,3)};

\draw[draw=white,fill=white] 
    plot[smooth, tension=1] coordinates{(0,0) (1,1.5) (0,3)};

 \draw[white,thick]
(0,0) -- (0,3);

\filldraw (0,3) circle (3pt)node[anchor=west]{$x$};
\filldraw (0,0) circle (3pt)node[anchor=north]{$y$};
\filldraw (-2,3) circle (3pt)node[anchor=east]{$v_4$};
\filldraw (-2,0) circle (3pt)node[anchor=north]{$v_5$};
\filldraw (-2,5) circle (3pt)node[anchor=south]{$v_3$};
\filldraw (0,5) circle (3pt)node[anchor=south]{$v_2$};

 \draw[black,thick]
(0,0) -- (-2,0) --  (-2,3) -- (0,3);

\draw[black,thick] (-2,3) -- (0,0);

\draw[black,thick] (-2,0) -- (0,3) -- (0,5) -- (-2,5) -- (0,3);

\draw[black,thick] (0,5) -- (-2,3) -- (-2,5);

\draw[thick] plot[smooth, tension=1] coordinates{(0,0) (1,1.5)
(0,3)}; \draw[thick] plot[smooth, tension=1] coordinates{(0,0)
(2.5,1.5) (0,3)};

  \node [rectangle, draw=white, fill=white] (b) at (-1,-2) {$(i)$};
8 \filldraw (1.7,1) circle (0pt)node[anchor=south]{$F_2$};




\draw[draw=gray!30!white,fill=gray!30!white] 
    plot[smooth, tension=1] coordinates{(11,0) (13.5,1.5) (11,3)}  -- 
    plot[smooth, tension=1] coordinates{(11,0) (13.5,1.5) (11,3)};

\draw[draw=white,fill=white] 
    plot[smooth, tension=1] coordinates{(11,0) (12,1.5) (11,3)};

 \draw[white,thick]
(11,0) -- (11,3);

\filldraw (11,3) circle (3pt)node[anchor=south]{$x$}; \filldraw
(11,0) circle (3pt)node[anchor=north]{$y$}; \filldraw (9,3) circle
(3pt); \filldraw (9,0) circle (3pt); \filldraw (7,0) circle (3pt);
\filldraw (7,3) circle (3pt);

 \draw[black,thick]
(7,3) -- (9,0) -- (7,0) -- (7,3) -- (9,3) -- (11,0) -- (9,0)  --
(9,3) -- (11,3);

\draw[black,thick] (7,0) -- (9,3);

\draw[black,thick] (9,0) -- (11,3);

\draw[thick] plot[smooth, tension=1] coordinates{(11,0) (12,1.5)
(11,3)}; \draw[thick] plot[smooth, tension=1] coordinates{(11,0)
(13.5,1.5) (11,3)};

  \node [rectangle, draw=white, fill=white] (b) at (10,-2) {$(ii)$};

\filldraw (12.7,1) circle (0pt)node[anchor=south]{$F_2$};


\draw[draw=gray!30!white,fill=gray!30!white] 
    plot[smooth, tension=1] coordinates{(19,0) (21.5,1.5) (19,3)}  -- 
    plot[smooth, tension=1] coordinates{(19,0) (21.5,1.5) (19,3)};

\draw[draw=white,fill=white] 
    plot[smooth, tension=1] coordinates{(19,0) (20,1.5) (19,3)};

 \draw[white,thick]
(19,0) -- (19,3);

\filldraw (19,3) circle (3pt)node[anchor=south]{$x$};
\filldraw (19,0) circle (3pt)node[anchor=north]{$y$};
\filldraw (18,4) circle (3pt);
\filldraw (18,1.5) circle (3pt);
\filldraw (18,-1) circle (3pt);
\filldraw (17,3) circle (3pt);
\filldraw (17,0) circle (3pt);

 \draw[black,thick]
(17,0) -- (18,1.5) -- (19,3) -- (18,4) -- (17,3) -- (19,0) -- (18,-1)
-- (17,0) -- (19,0);

 \draw[black,thick]
(19,3) -- (17,3);

 \draw[black,thick]
(18,4) -- (18,1.5) -- (18,-1);

\draw[thick] plot[smooth, tension=1] coordinates{(19,0) (20,1.5)
(19,3)}; \draw[thick] plot[smooth, tension=1] coordinates{(19,0)
(21.5,1.5) (19,3)};

\filldraw (20.7,1) circle (0pt)node[anchor=south]{$F_2$};

  \node [rectangle, draw=white, fill=white] (b) at (19,-2) {$(iii)$};

\end{tikzpicture}
\caption{The graph $G$ in the subcases $G_1=H_1$ and $G_1=H_2$ of Subcase 3.2.} \label{fig:H_1new_app}
\end{figure}
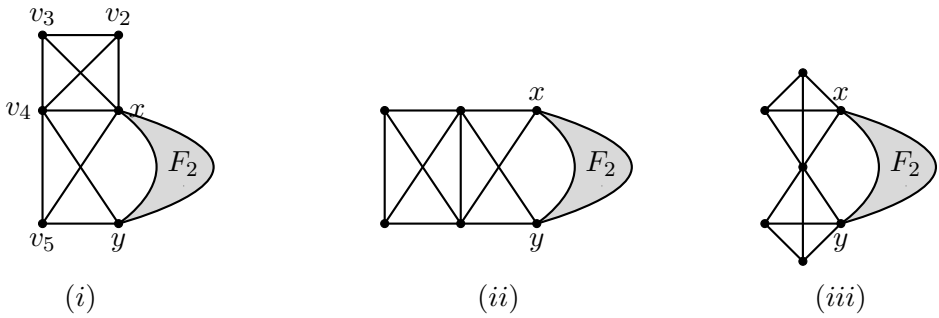

Suppose $G_1=H_1$. Since $F_1$ is an atom,  $G$ would have the
structure of one of the graphs shown in
 Figure \ref{fig:H_1new_app}(i) and (ii).  Since $v$ has degree at least four in $G_1$, case (ii) cannot occur and $v=x$ when
 case (i) occurs.
 Consider the realisation $(G,p)$ where $p|_{F_1}$ is given by $p(y)=(0,1,0)$, $p(x)=(x_1,y_1,z_1)$,  $p(v_2)=(-1,0,-1/3)$, $p(v_3)=(1/\sqrt{2},-1/\sqrt{2},1/3)$, $p(v_4)=(1,0,-1)$ and $p(v_5)=(0,-1,z_5)$.
For simplicity in the expressions we also specify that $x_1=-1/\sqrt{2},y_1=1/\sqrt{2}$ and $z_1=-2$.
Then the unique vector $ m_1(v)=(1,u_5,u_6)$ with $m_1\in Z_1$ is given by
$u_5=\frac{(49\sqrt{2}+73)(17z_5-71+77\sqrt{2})}{17(31z_5+111+21\sqrt{2})}$
and
$u_6=\frac{(-1+4\sqrt{2})(3z_5+17+11\sqrt{2})}{31z_5+111+21\sqrt{2}}.$
By putting $z_5=0,1$ in the above expressions for $(1,u_5,u_6)$ we get the infinitesimal velocities:
$$m_1(v)=\left(1,\frac{(49\sqrt{2}+73)(-71+77\sqrt{2})}{17(111+21\sqrt{2})}, \frac{(-1+4\sqrt{2})(17+11\sqrt{2})}{111+21\sqrt{2}}\right)$$ and
$$\tilde m_1(v)=\left(1, \frac{(49\sqrt{2}+73)(-54+77\sqrt{2})}{17(142+21\sqrt{2})}, \frac{(-1+4\sqrt{2})(20+11\sqrt{2})}{142+21\sqrt{2}}\right).$$

We also consider the realisation $(G,\hat p)$ where $\hat p|_{F_2}=p|_{F_2}$ and $\hat p|_{F_1-\{x,y\}}$ is given by
$\hat p(v_2)=(1/\sqrt{2},-1/\sqrt{2},-1/2)$, $\hat p(v_3)=(-1,0,1/2)$, $\hat p(v_4)=(1/\sqrt{2},1/\sqrt{2},1)$ and $\hat p(v_5)=(1,0,z_5)$.
For simplicity in the expressions we again specify that $x_1=-1/\sqrt{2}$, $y_1=1/\sqrt{2}$ and $z_1=-2$.
Then the unique vector $\hat m_1(v)=(1,u_5,u_6)$ with $\hat m_1\in Z_1$ is given by
$u_5=\frac{\sqrt{2}z_5+2z_5^2-9\sqrt{2}+11z_5-12}{4\sqrt{2}z_5+2z_5^2-15\sqrt{2}+17z_5-18}$ and
$u_6=-\frac{5z_5-12+\sqrt{2}z_5-3\sqrt{2}+2z_5^2}{4\sqrt{2}z_5+2z_5^2-15\sqrt{2}+17z_5-18}$.
By putting $z_5=0$ in the this expression for $\hat m_1(v)$ we get the infinitesimal velocity
$$\hat m_1(v)=\left(1, \frac{-12-9\sqrt{2}}{-18-15\sqrt{2}}, -\frac{-12-3\sqrt{2}}{-18-15\sqrt{2}}\right).$$
We can now check that $m_1(v)$, $\tilde m_1(v)$ and $\hat m_1(v)$ are linearly independent.

The only remaining alternative is $G_1=H_2$. Since $F_1$ is an atom,
$G$ would have the structure of the graph shown in Figure
\ref{fig:H_1new_app}(iii). By symmetry we may assume that $v=x$. We
can now use Lemma \ref{lem:vfree}(d), with $H$ equal to the copy of
$K_4$ attached at $y$, and induction to deduce that $G$ is $v$-free
rigid.
\end{proof}

We close this section by
using Theorem \ref{thm:vfree} to obtain a full characterisation of
$v$-free rigidity on the cylinder.

\begin{thm}\label{thm:vfreefull}
Let $G=(V,E)$ be a graph and $v\in V$. Then $G$ is $v$-free
rigid on $\Y$ if and only if $G$ is rigid on $\Y$
and $v$ is contained
in an $\M_{2,2}^*$-circuit of $G$.
\end{thm}

\begin{proof}
Suppose $G$ is infinitesimally $v$-free rigid and choose a
minimally $v$-free rigid spanning subgraph $H=(V,E)$ of $G$. Then
$H$ is rigid on $\Y$ and  $|E|=2|V|-1$, so $H$ contains a unique
$\M_{2,2}^*$-circuit $C$. If $v\not\in V(C)$ then the rows of
$R_v(H,p)$ indexed by $C$ would be linearly dependent for any
generic $(H,p)$ and hence $H$ would not be minimally $v$-free rigid.
Hence $v\in V(C)$.

Conversely, suppose that $G=(V,E)$ is rigid on $\Y$
and $v$ is contained
in an $\M_{2,2}^*$-circuit $C$ of $G$. We will show that $G$ is
infinitesimally $v$-free rigid by induction on $|V|+|E|$. Choose a
spanning subgraph $H$ of $G$ such that $H$ is rigid on $\Y$,
$C\subseteq H$ and $H$ has as few edges as possible. If $H\neq G$
then we are through by induction, so we may assume that $H=G$. Hence
$|E|=2|V|-1$ and $C$ is the unique $\M_{2,2}^*$-circuit in $G$.

Let $H_1$ be a maximal $v$-free rigid subgraph of $G$ with
$C\subseteq H_1$.  Note that $H_1$ exists since $C$ itself is
$v$-free rigid by Theorem \ref{thm:vfree}. We may assume that
$V(G-H_1)\neq \emptyset$. The facts that $|E|=2|V|-1$ and
$|E(H_1)|=2|V(H_1)|-1$ now imply that there exists a vertex $u$ of
degree two or three in $G$ which does not belong to $H_1$. If $u$
has degree two then we can apply induction to $G-u$ and then use
Lemma \ref{lem:vfree}(a) to deduce that $G$ is $v$-free rigid. Hence
we may assume that $d(u)=3$ and let $N(u)=\{x,y,z\}$. Since $G$ has
a unique $\M_{2,2}^*$-circuit, at most two of $x,y,z$ are contained
in $V(H_1)$.

Suppose $u,x,y,z$ induces a $K_4$ in $G$.
Let
$H_2$ be a maximal subgraph of $G$ such that $u,x,y,z\in V(H_2)$ and
$|E(H_2)|=2|V(H_2)|-2$.
Then $H_2$ is rigid on $\Y$ by Theorem \ref{thm:full}. If $H_1\cap
H_2\neq \emptyset$ then $H_1\cup H_2$ is $v$-free rigid by  Lemma
\ref{lem:vfree}(c). This contradicts the maximality of $H_1$ and
hence $H_1\cap H_2= \emptyset$. Since $G/H_2$ is $v$-free rigid on
$\Y$ by induction, we can now use Lemma \ref{lem:vfree}(d) to deduce
that $G$ is $v$-free rigid on $\Y$.

Hence we may suppose that $u,x,y,z$ does not induce a $K_4$ in $G$.
If we can apply a 1-reduction at $u$ to obtain a graph $G'$ which
has $|E(G')|=2|V(G')|-1$ and contains a unique $\M_{2,2}^*$-circuit,
then we can apply induction and Lemma \ref{lem:vfree}(b) to deduce
that $G$ is $v$-free rigid on $\Y$. Hence we may assume that, for
all $a,b\in N(u)$ with $\{a,b\}\not\subset V(C)$, either $ab\in E$
or there exists a 
set $X_{ab}\subseteq V-u$ which induces $2|X_{ab}|-2$ edges in $G$ and has $a,b\in X_{ab}$.

Suppose at most one of $x,y,z$ is in $V(C)$. Since $u,x,y,z$ does
not induce a $K_4$ in $G$ we may assume that $X_{xy}$ exists. Since
$C$ is the unique $\M_{2,2}^*$-circuit in $G$, $z\not\in X_{xy}$. If
$\{xz,yz\}\subset E$ then $G[X_{xy}\cup \{u,z\}]$ would contain an
$\M_{2,2}^*$-circuit distinct from $C$. Hence we may suppose that
$xz\notin E$ and $X_{xz}$ exists. Then $G[X_{xy}\cup X_{xz}\cup
\{u\}]$ contains an $\M_{2,2}^*$-circuit distinct from $C$.

We may now assume that $x,y \in V(C)$ and $z\not\in V(C)$. If
$\{xz,yz\}\subseteq E$ then $(C-e)\cup G[u,x,y,z]$ would contain an
$\M_{2,2}^*$-circuit distinct from $C$ for any $e\in E(C)$. Hence we
may assume that $X_{xz}$ exists. Then $(C-e)\cup G[X_{xz}\cup
\{u\}]$ again contradicts the uniqueness of $C$.
\end{proof}

\section{Vertically restricted frameworks}
\label{sec:restricted}

In this section we prove Lemma \ref{lem:zrestricted2}. We will assume throughout the section that all frameworks are in standard position on $\Y$, with respect to a fixed vertex $v_1$. Given a framework $(G,\hat p)$ on $\Y$, we are interested in the set of frameworks $(G,p)$  on $\Y$ which are equivalent to $(G,\hat p)$ and satisfy $(z_1,z_2,\dots,z_n)=b(\hat z_1,\hat z_2,\dots,\hat z_n)$ for some $b\in \mathbb{R}$, where $\hat p(v_i)=(\hat x_i,\hat y_i,\hat z_i)$ and $p(v_i)=(x_i,y_i,z_i)$. We will say that such a framework $(G,p)$ is \emph{vertically restricted equivalent} to $(G,\hat p)$. Henceforth we will abbreviate `vertically restricted' by the prefix VR.
We say that two frameworks are \emph{VR-congruent} if they are linked by a
a reflection in the $xy$-plane or the $yz$-plane.
Thus VR-congruence imples VR-equivalence. We say that $(G,\hat p)$ is \emph{globally VR-rigid} if every VR-equivalent framework is VR-congruent to it.
The set of all frameworks which are VR-equivalent to $(G,\hat p)$ is given by the set of solutions to the following system of equations:
\begin{eqnarray}
\label{eqn:a} \|p(v_i)-p(v_j)\|^2&=& c_{ij} \hspace{3cm} (v_iv_j\in E) \\
\label{eqn:b}  \frac{z_i}{z_2}&=& k_{i} \hspace{3.1cm} (v_i\in V\setminus \{v_1,v_2\})\\
\label{eqn:c} x_i^2+y_i^2&=& 1 \hspace{3.3cm} (v_i\in V-v_1)
\end{eqnarray}
for constants $c_{ij}= \|\hat p(v_i)-\hat p(v_j)\|^2$ and $k_i=\frac{\hat z_i}{\hat z_2}$, assuming $\hat z_2\neq 0$.
We can differentiate these equations to obtain the following linear system for the unknowns $\dot p(v_i)$, $v_i\in V-v_1$:
\begin{eqnarray}
\label{eqn:1} (p(v_i)-p(v_j))\cdot (\dot p(v_i)-\dot p(v_j))&=& 0 \hspace{3cm} (v_iv_j\in E) \\
\label{eqn:2}  z_2\dot z_i-z_i\dot z_2&=& 0 \hspace{3cm} (v_i\in V\setminus \{v_1,v_2\})\\
\label{eqn:3} x_i\dot x_i+y_i\dot y_i&=& 0 \hspace{3cm} (v_i\in V-v_1).
\end{eqnarray}

We say that $(G,\hat p)$ is \emph{infinitesimally VR-rigid} if the only solution to this linear system is the trivial solution $\dot p(v_i)=0$ for all $v_i\in V-v_1$,
or equivalently if the rank of the matrix of coefficients of the system is $3|V|-3$.
This matrix,
the {\em VR-rigidity matrix of $(G,p)$}, has $3|V|-3$ columns and $|E|+2|V|-3$ rows. The rows corresponding to (\ref{eqn:1}) have the form
$$\begin{pmatrix}
\dots & 0 & p(v_i)-p(v_j) & 0 & \dots & 0 & p(v_j)-p(v_i) & 0 & \dots
\end{pmatrix}
$$
when $v_i,v_j\neq v_1$, and
$$\begin{pmatrix}
\dots & 0 & 0 & 0 & \dots & 0 & p(v_j)-p(v_1) & 0 & \dots
\end{pmatrix}
$$
when $v_i=v_1$.
The rows corresponding to (\ref{eqn:2}) have the form
$$\begin{pmatrix}
(0,0,-z_i) & 0 &\dots & 0 & (0,0,z_2) & 0 & \dots
\end{pmatrix}. $$
The rows corresponding to (\ref{eqn:3}) have the form
$$\begin{pmatrix}
\dots & 0 & (x_i,y_i,0) & 0 & \dots
\end{pmatrix}. $$

We say that the graph $G$ is \emph{VR-rigid} if $(G,p)$ is {infinitesimally VR-rigid} for some, or equivalently every, quasi-generic framework $(G,p)$ in standard position on $\Y$.  The graph $G$ is
\emph{minimally VR-rigid}
if $G$ is VR-rigid, and $G-e$ is not rigid for all edges $e$ of $G$.
We first characterise minimal VR-rigidity.

\begin{thm}\label{thm:vrmin}
A graph $G$ is minimally VR-rigid if and only it is connected and contains exactly one cycle.
\end{thm}

\begin{proof}
It is straightforward to show that if $G$ is minimally VR-rigid then $G$ is connected and has $|V|$ edges.
To prove the converse, we first observe that every connected graph
with exactly one cycle can be constructed recursively from $K_3$ by
subdividing edges and adding vertices of degree one. The
structure of the VR-rigidity matrix implies that $K_3$ is
minimally VR-rigid and that adding vertices of degree one  preserves
minimal rigidity. To see that subdivision preserves minimal rigidity
we can apply a similar technique to the proof of Lemma
\ref{lem:vfree}(b). (The new proof is simpler since there is no free vertex and our restiction to frameworks in standard position  allows no
infinitesimal isometry of the cylinder.)
\end{proof}


The remaining results of this section use the algebraic independence of coordinates of a quasi-generic framework in standard position on $\Y$. The proof technique is similar to that in \cite{JJS, JMN}. In particular we use the following analogue of the `$\Y$-rigidity map' from \cite{JMN}. Given a graph $G=(V,E)$ with $n$ vertices and $m$ edges, we define three maps with domain $D= \{(x_2,y_2,z_2,\ldots,x_n,y_n,z_n)\in \bR^{3n-3}\,:\,z_2\neq 0\}$. We first associate a framework $(G,p)$ in standard position on $\Y$ with each $\tilde p \in D$ by putting $p=(0,1,0,\tilde p)$. The first map  $f_G:D\rightarrow \bR^{m}$ is defined by $f_G(\tilde p)=(\|e_1 \|^2, \dots, \|e_{m}\|^2)$ where $\|e_i\|^2=\|p(v_j)-p(v_k)\|^2$
for each $e_i=v_jv_k\in E$. The second map $h_G:D\rightarrow \bR^{n-2}$ is given by $h_G(\tilde p)=(z_3/z_2,z_4/z_2,\dots, z_n/z_2)$. The third map $\theta_G:D\rightarrow \bR^{n-1}$ is given by $\theta_G(\tilde p)=(x_2^2+y_2^2,\dots, x_n^2+y_n^2)$. The {\em VR-rigidity map} $F_G^{\VR}:D\rightarrow \bR^{m+2n-3}$ is defined by $F^{\VR}_G=(f_G,h_G,\theta_G)$. We showed in \cite[Lemma 8]{JMN} that a framework $(G,p)$ is quasi-generic and in standard position on $\Y$ if and only if  $\td[\bQ(p):\bQ]=2n-2$. We may  use a similar proof technique to \cite[Lemma 9]{JMN} to show further, that if $G$ is VR-rigid and $(G,p)$ is quasi-generic and in standard position on $\Y$, then
$\overline{\bQ(p)}=\overline{\bQ(f_G(p),h_G(p))}$. This gives

\begin{lem}\label{lem:closures}
Suppose $(G,p)$ and $(G,q)$ are two VR-equivalent frameworks in standard position on $\Y$ and that $(G,p)$ is infinitesimally VR-rigid and quasi-generic. Then $\overline{\bQ(p)}=\overline{\bQ(f_G(p),h_G(p))}=\overline{\bQ(q)}$ and $\td[\bQ(p):\bQ]=2|V|-2$.
\end{lem}

We say that a graph $G$ is \emph{redundantly VR-rigid} on $\Y$ if $G-e$ is VR-rigid on $\Y$ for all edges $e$ of $G$. Our next result shows that 2-connectivity and redundant rigidity are necessary conditions for generic global VR-rigidity.

\begin{lem}\label{lem:redundant}
Let $(G,p)$ be a quasi-generic, globally VR-rigid framework in standard position on $\Y$. Then $G$ is $2$-connected and redundantly VR-rigid on $\Y$.
\end{lem}

\begin{proof}
Suppose $G$ is not 2-connected. Then we have $G = G_1 \cup G_2$ for
subgraphs $G_i =(V_i,E_i)$ with $|V_1\cap V_2|\leq 1$ and $v_1\in V_1$. Let $p_1 =
p|_{V_1}$ and $p_2 = p|_{V_2}$. Let $(G,q)$ be obtained from $(G,
p)$ by reflecting $(G_2, p_2)$ in a plane which contains $p(V_1 \cap
V_2)$ and also contains the z-axis. Then $(G, q)$ is a framework in standard position on
$\Y$ and is VR-equivalent but not VR-congruent to $( G , p )$.

Suppose  $G-e$ is not VR-rigid for some $e=v_iv_j\in E$.
Since $p$ is quasi-generic, we can use a similar argument to
that given in \cite{JMN} to deduce that the configuration space of
all frameworks in standard position on $\Y$ which are
VR-equivalent to $(G-e,p)$ is a 1-dimensional manifold. Let $\C$ be
the component of  this manifold which contains $p$.
Then $\C$ is bounded since $G-e$ is VR-rigid and hence connected.
The fact that $\C$ is closed now implies that $\C$ is diffeomorphic
to a circle.
We can now use a similar argument to that given in \cite[Theorem 9]{JMN} to find a $q\in \C$ such that $(G,q)$ is VR-equivalent, but not VR-congruent, to $(G,p)$.
\end{proof}

\begin{thm}\label{thm:vrglobal}
A quasi-generic framework $(G,p)$ in standard position on $\Y$ is globally VR-rigid if and only if $G=(V,E)$ is 2-connected and $|E|\geq |V|+1$.
\end{thm}

\begin{proof}
Necessity follows from Lemma \ref{lem:redundant}. To prove the converse,
we assume that $(G,q)$ is a quasi-generic framework in standard position on $\Y$ which is VR-equivalent to $(G,p)$. Let $p(v_i)=(x_i,y_i,z_i)$ and $q(v_i)=(a_i,b_i,c_i)$ for all $v_i\in V$.

We first consider the case when $G$ contains a vertex $v_0$ of degree two which is distinct from $v_1$. Let $v_j,v_k$ be the neighbours of $v_0$ in $G$. By symmetry we may assume that $v_j\neq v_1$.
We have the following constraints on the positions of $v_0,v_j,v_k$:
\begin{eqnarray}
\label{eqn6} (x_0-x_j)^2+(y_0-y_j)^2+(z_0-z_j)^2-(a_0-a_j)^2-(b_0-b_j)^2-(c_0-c_j)^2&=&0\\
\label{eqn7}(x_0-x_k)^2+(y_0-y_k)^2+(z_0-z_k)^2-(a_0-a_k)^2-(b_0-b_k)^2-(c_0-c_k)^2&=&0\\
\label{eqn8}-c_0z_j+c_jz_0&=&0
\end{eqnarray}
and
\begin{equation}\label{eqn10}
x_i^2+y_i^2=1=a_i^2+b_i^2 \mbox{ for all $i\in\{0,j,k\}$}.
\end{equation}

We can  use a computer algebra package
to solve
equations (\ref{eqn6})-(\ref{eqn8}) for $a_0,b_0$ and $c_0$ and then
substitute these values into 
the equation $a_0^2+b_0^2=1$
to obtain an equation of
the form $P=0$ where $P=\sum_{i=0}^8 t_iz_0^i$,
$t_i\in \overline{\bQ(p|_{G'},q|_{G'},x_0,y_0)}$ and $G'=G-v_0$.  Since $G'$ is infinitesimally VR-rigid by Theorem \ref{thm:vrmin}, we have $\overline{\bQ(q|_{G'})}=\overline{\bQ(p|_{G'})}$ by Lemma \ref{lem:closures}, and hence $t_i\in \overline{\bQ(p|_{G'},x_0,y_0)}$. Since $z_0$ is algebraically independent over $\overline{\bQ(p|_{G'},x_0,y_0)}$, we must have $t_i=0$ for $0\leq i\leq 8$. We may use equations 
(\ref{eqn10})
to simplify the expression for $t_8$ to give   
$$t_8=-4(c_j-z_j)^4(c_j+z_j)^4(a_ja_k-b_jb_k+1)(a_ja_k+b_jb_k-1)=0.$$ 
Since $(G',q|_{G'})$ is quasi-generic and in standard position with respect to $v_1$, we have $(a_ja_k-b_jb_k+1)(a_ja_k+b_jb_k-1)\neq 0$.
 Hence $c_j=\pm z_j$. By reflecting $(G,q)$ in the $xy$-plane, we may assume that $c_j=z_j$. The fact that $(G,p)$ and $(G,q)$ are VR-equivalent now gives $c_i=z_i$ for all $v_i\in V$.
This in turn implies that the frameworks $(G,\tilde p)$ and $(G,\tilde q)$ obtained by projecting $(G,p)$ and $(G,q)$ onto the $xy$-plane are equivalent. Since  quasi-generic realisations of 2-connected graphs on the unit circle are globally rigid, $(G,\tilde p)$ and $(G,\tilde q)$ are congruent on the unit circle. Since $\tilde p(v_1)=(0,1)=\tilde q(v_1)$, either $\tilde p=\tilde q$, or $\tilde p$ can be obtained from $\tilde q$ by reflection in the $y$-axis. This implies that 
either $p=q$, or $p$ can be obtained from $q$ by reflection in the $yz$-plane. Hence $(G,p)$ and $(G,q)$ are VR-congruent.

It remains to consider the case when all vertices of $G$ other than $v_1$ have degree at least three. In this case it is easy to see that  $G-e$ is a 2-connected graph with at least $|V|+1$ edges for some $e\in E$. We may now use induction to deduce that $(G-e,p)$ is globally VR-rigid and hence that
$(G,q)$ is VR-congruent to $(G,p)$. 
\end{proof}

Lemma \ref{lem:zrestricted2} is just a restatement of Theorem \ref{thm:vrglobal}. 

\section{Closing Remarks}

1) We can also consider frameworks on families of concentric
cylinders.  Let $(\mathcal{Y}_1, \mathcal{Y}_2, \dots, 
\mathcal{Y}_n)$ be an ordered family of (not necessarily distinct)
concentric cylinders  where $\mathcal{Y}_i=\{(x,y,z)\in
\mathbb{R}^3:x^2+y^2=r_i\}$ and $r=(r_1,\dots,r_n)$ is a vector of
positive real numbers.
We say that the family is {\em generic} if the radii $r_i$ are
algebraically independent over $\bQ$. A \emph{framework  on $(\mathcal{Y}_1, \mathcal{Y}_2, \dots, 
\mathcal{Y}_n)$} is an ordered pair $(G,p)$ consisting of a graph $G$  and a
realisation $p$ such that $p(v_i)\in \mathcal{Y}_i$ for all $v_i\in
V$.

In \cite{NOP} it was proved that a graph is generically rigid on a
family of concentric cylinders if and only if it is generically
rigid  on the cylinder. We may adapt the proof of Theorem
\ref{thm:full} to obtain a similar result for generic global
rigidity on families of cylinders, but our proof only works for {\em
generic} families of cylinders. Given a framework $(G,p)$ in $\bR^3$ with $p(v_i)=(x_i,y_i,z_i)$, we define the family of concentric
cylinders induced by $p$ to be $\Y^p=(\mathcal{Y}_1, \mathcal{Y}_2, \dots, 
\mathcal{Y}_n)$   where $\mathcal{Y}_i=\{(x,y,z)\in
\mathbb{R}^3:x^2+y^2=x_i^2+y_i^2\}$.


\begin{thm}\label{thm:concentric}
Let $(G,p)$ be a generic framework in $\mathbb{R}^3$.  Then $(G,p)$
is globally rigid on the family $\Y^p$ of concentric cylinders
induced by $p$ if and only if $G$ is either a complete graph on at
most four vertices or $G$ is 2-connected and redundantly rigid on
the cylinder.
\end{thm}

\begin{proof}
Necessity follows from \cite{JMN}. The proof of sufficiency is
similar to the proof of Theorem \ref{thm:full}.
Theorem \ref{thm:vsplitmax} and \cite[Lemma 10]{J&N} (which tells us
that, if $G$ has a realisation with a maximum rank equilibrium
stress on some family of cylinders, then every generic realisation
of $G$ on a generic family of cylinders will have the same property)
imply that $(G,p)$ has a maximum rank equilibrium stress on $\Y^p$ when $G$ is
an $\M_{2,2}^*$-circuit. Combining this with \cite[Theorem 9]{J&N}
(which tells us that $(G,p)$ is globally rigid on $\Y^p$ when it has
a maximum rank equilibrium stress) implies that $(G,p)$ is globally
rigid on $\Y^p$ when $G$ is an  $\M_{2,2}^*$-circuit. The extension
from a circuit to a 2-connected redundantly rigid graph can be
accomplished using an ear decomposition of the rigidity matroid as
in the proof of Theorem \ref{thm:full}.
\end{proof}

2) As noted in the proof of Theorem  \ref{thm:concentric},
\cite[Theorem 9]{J&N} tells us that a generic framework on a generic
family of concentric cylinders is globally rigid if it has a maximum
rank equilibrium stress. Here we show that the same result holds for
the cylinder.

\begin{thm}\label{thm:unitcylinderglobal}
Let $(G,p)$ be a generic framework on $\Y$ with a
maximum rank equilibrium stress $(\omega,\lambda)$. Then $(G,p)$ is
globally rigid on $\Y$.
\end{thm}

\begin{proof}
The result will follow from Lemmas \ref{lem:zrestricted} and
\ref{lem:zrestricted2} if we can show that $G$ is 2-connected and
has at least $n+1$ edges, where $n=|V(G)|$.

 We first show that $G$
is 2-connected. Suppose not, then $G = G_1 \cup G_2$ for subgraphs
$G_i =(V_i,E_i)$ with $|V_1\cap V_2|\leq 1$. Let
$\omega^i=\omega|_{E_i}$. Then $\omega^i$ can be extended to a
unique equilibrium stress $(\omega^i,\lambda^i)$ for $(G_i,p|_{G_i})$ on $\Y$,
for $i=1,2$. We will consider the case when $|V_1\cap V_2|=1$ (the
case when $V_1\cap V_2=\emptyset$ is simpler). Let
$V_1=\{v_1,v_2,\ldots,v_k\}$ and $V_2=\{v_k,v_{k+1},\ldots,v_n\}$.
Then
$$\Omega(\omega)=\begin{pmatrix} A_1 & a^T & 0
\\ a & b& c   \\ 0 & c^T & A_2
\end{pmatrix}, \mbox{ where }
\Omega(\omega_1)=\begin{pmatrix} A_1 & a^T\\ a & b_1
\end{pmatrix}, \:\: \Omega(\omega_2)=\begin{pmatrix} b_2 & c \\
c^T & A_2\end{pmatrix}$$ and $b_1+b_2=b$.

Since $\Omega(\omega_1)$ has column sums equal to zero, we can use
elementary row operations to reduce $\Omega(\omega)$ to the form
$$\begin{pmatrix} A_1 & a^T & 0
\\ 0 & b_2& c   \\ 0 & c^T & A_2
\end{pmatrix}.$$
Since $\rank \begin{pmatrix} A_1 & a^T
\end{pmatrix} = \rank \Omega(\omega_1)\leq |V_1|-2$ and
$\rank \begin{pmatrix}  b_2& c   \\ c^T & A_2
\end{pmatrix} =\rank \Omega(\omega_2)\leq |V_2|-2$, we have $\rank
\Omega(\omega)\leq n-3$. This contradicts the hypothesis that
$(\omega,\lambda)$ is a maximum rank equilibrium stress.

We next show $G$ has at least $n+1$ edges. In fact we will show the
stronger fact that $G$ is rigid on $\Y$. Since deleting an edge $e$
with $\omega_e=0$ will not reduce the rank of
$\Omega_\CYL(\omega,\lambda)$, we may assume that $\omega_e\neq 0$
for all edges $e$ of $G$.  Then every edge of $G$ is contained in a
$\M_{2,2}^*$-circuit. Since $G$ is connected, $\M_{2,2}^*$-circuits
are rigid, and the union of any two rigid subgraphs with a non-empty
intersection is rigid, $G$ is rigid on $\Y$.
\end{proof}

3) We conjecture that the converse to Theorem
\ref{thm:unitcylinderglobal} holds. That is, a generic framework
$(G,p)$ which is globally rigid on  $\Y$ must have
a maximum rank equilibrium stress. This would be an analogue of the
main result of \cite{GHT}. It follows from our work that this
conjecture holds when $G$ is an $\M_{2,2}^*$-circuit since we have
shown that every $\M_{2,2}^*$-circuit is both globally rigid and has
a maximum rank equilibrium stress. One way to extend this to all
graphs would be to prove a version of Lemma \ref{lem:glue} showing
that the `glueing' operation preserves the property of having a
maximum rank equilibrium stress. We can adapt the proof of Connelly
\cite[Lemma 10]{Ccomb} to prove the special case when $G_1$ and
$G_2$ have exactly two vertices in common but the general case seems
challenging. An alternative approach to this problem would be to try
to extend the recursive characterisation of $\M_{2,2}^*$-circuits to
a recursive characterisation of graphs which are 2-connected and
redundantly rigid on $\Y$, and then use it to show that all generic
realisations of these graphs have a maximum rank equilibrium stress.

4) In \cite{NSTW} frameworks on `expanding' spheres were considered.
Our $v$-free rigidity model is analogous to this for the cylinder.
We believe that a similar technique to the proof of Theorem
\ref{thm:vfreefull} could be used to extend  our characterisation of
$v$-free rigidity on $\Y$  to the case when a subset of the vertices
is allowed to move off the cylinder in a coordinated way. Since the
analogous problem when all vertices are free to move off the
cylinder independently is exactly the well known 3-dimensional
rigidity problem, it may be of interest to extend this further to
consider the case when two or more subsets of the vertex set can
move independently off the cylinder. It is not hard to adapt the
examples given in \cite{NSTW} to show that the obvious sparsity
count is not sufficient to imply minimal rigidity when there are at
least three subsets moving independently, but it may be sufficient
when there are at most two subsets.

\section*{Appendix: Base graphs for the proofs of Theorems \ref{thm:unitstress} and \ref{thm:vfree}}
We define a framework $(G,p)$ for $G\in \{K_5^-, H_1,H_2\}$  which
is infinitesimally rigid on $\Y$ and has a nowhere
zero maximum rank equilibrium stress $(\omega,\lambda)$.
These constructions form the base case for the inductive proofs of Theorems \ref{thm:unitstress} and \ref{thm:vfree}. We will use the labeling of the vertices given in Figure \ref{fig:smallgraphs} and adopt the convention that
 $\omega_{ij}$  is the weight on the edge $v_iv_j$ in $\omega$ and  $\lambda_i$  is the weight on the vertex
 $v_i$ in $\lambda$.
\\

\subsection*{\bf Case 1: $G=K_5^-$} Let $(G,p)$ and
$(\omega,\lambda)$ be defined by $p(v_1)=(0,1,0), p(v_2)=(1,0,-1),
p(v_3)=(\frac{1}{\sqrt{2}},-\frac{1}{\sqrt{2}},\frac{1}{3}),
p(v_4)=(-1,0,-\frac{1}{3}),
p(v_5)=(\frac{1}{\sqrt{2}},\frac{1}{\sqrt{2}},\frac{1}{2})$,
\begin{multline*}
(\omega_{12},\omega_{13},\omega_{14},\omega_{15},\omega_{23},\omega_{24},\omega_{25},\omega_{34},\omega_{45})=
(239, -216-654\sqrt{2}, 201+270\sqrt{2}, 756+616\sqrt{2},\\
 108+327\sqrt{2}, -\frac{1635}{2}-852\sqrt{2},108+88\sqrt{2}, -108-327\sqrt{2}, -648-528\sqrt{2})
\end{multline*}
and
$$(\lambda_1,\lambda_2,\lambda_3,\lambda_4,\lambda_5)=(290+254\sqrt{2},1595+1397\sqrt{2},1524+870\sqrt{2}, 3045+2667\sqrt{2},1016+580\sqrt{2}).$$
It is straightforward to check that $\rank R_\CYL(G,p)=13$, that
$(\omega,\lambda)\cdot R_\CYL(G,p)=0$ and that $\rank
\Omega_\CYL(\omega,\lambda)=9$.

\subsection*{\bf Case 2: $G=H_1$} Let $(G,p)$ and
$(\omega,\lambda)$ be defined by $p(v_1)=(0,1,0),
p(v_2)=(-1,0,-\frac{1}{3}),
p(v_3)=(\frac{1}{\sqrt{2}},-\frac{1}{\sqrt{2}},\frac{1}{3}),p(v_4)=(1,0,-1),
p(v_5)=(0,-1,\frac{2}{3}),
p(v_6)=(\frac{1}{\sqrt{2}},\frac{1}{\sqrt{2}},\frac{1}{2})$,
\begin{multline*}
(\omega_{12},\omega_{13},\omega_{14},\omega_{15},\omega_{16},\omega_{23},\omega_{24},\omega_{34},\omega_{45},\omega_{46},\omega_{56})=
(1,2\sqrt{2},-\frac{361}{441}+\frac{10}{49}\sqrt{2},-\frac{62}{147}-\frac{82}{147}\sqrt{2},\\
-\frac{20}{49}-\frac{80}{441},\sqrt{2},\frac{1}{2}+\sqrt{2},-\sqrt{2},\frac{32}{147}+\frac{12}{49}\sqrt{2},\frac{4}{49}+\frac{16}{441}\sqrt{2},-\frac{24}{49}-\frac{32}{147}\sqrt{2})
\end{multline*}
and
$$(\lambda_1,\lambda_2,\lambda_3,\lambda_4,\lambda_5,\lambda_6)=(-\frac{10}{9}-\frac{10}{9}\sqrt{2}, -3-3\sqrt{2}, -4-2\sqrt{2}, -\frac{13}{9}-\frac{13}{9}\sqrt{2}, \frac{4}{3}+\frac{4}{3}\sqrt{2}, \frac{8}{9}+\frac{4}{9}\sqrt{2}).$$
It is straightforward to check that $\rank R_\CYL(G,p)=16$, that
$(\omega,\lambda)\cdot R_\CYL(G,p)=0$  and that $\rank
\Omega_\CYL(\omega,\lambda)=12$.

\subsection*{\bf Case 3: $G=H_2$} Let $(G,p)$ and
$(\omega,\lambda)$ be defined by $p(v_1)=(0,1,0),
p(v_2)=(-1,0,-\frac{1}{3}),
p(v_3)=(\frac{1}{\sqrt{2}},-\frac{1}{\sqrt{2}},\frac{1}{3}),p(v_4)=(1,0,-1),
p(v_5)=(0,-1,\frac{2}{3}),
p(v_6)=(\frac{1}{\sqrt{2}},\frac{1}{\sqrt{2}},\frac{1}{2})$,
$p(v_7)=(-\frac{1}{\sqrt{2}},-\frac{1}{\sqrt{2}},-\frac{1}{4})$
\begin{multline*}
(\omega_{12},\omega_{13},\omega_{14},\omega_{17},\omega_{23},\omega_{24},\omega_{34},\omega_{45},\omega_{46},\omega_{47},\omega_{56},\omega_{57},\omega_{67})=
(1,2\sqrt{2}, \frac{4}{7}\sqrt{2}-\frac{13}{21},\frac{8}{7}+\frac{8}{21}\sqrt{2},\\
\sqrt{2},\frac{1}{2}+\sqrt{2},-\sqrt{2}, -\frac{2652}{25165}-\frac{338}{25165}\sqrt{2}, -\frac{2764}{25165}\sqrt{2}-\frac{4652}{25165}, \frac{28424}{75495}\sqrt{2}+\frac{24784}{25165},\\
 \frac{2764}{25165}\sqrt{2}+\frac{4652}{25165}, \frac{568}{3595}+\frac{16}{3595}\sqrt{2}, \frac{18608}{45297}+\frac{11056}{45297}\sqrt{2})
\end{multline*}
and
\begin{multline*}
(\lambda_1,\lambda_2,\lambda_3,\lambda_4,\lambda_5,\lambda_6,\lambda_7)=(-\frac{74}{21}\sqrt{2}-\frac{82}{21},-3-3\sqrt{2},-2\sqrt{2}-4,\frac{269}{105}-\frac{253}{105}\sqrt{2}, -\frac{4}{35}\sqrt{2}-\frac{12}{35},\\
-\frac{212}{315}\sqrt{2}-\frac{328}{315}, -\frac{704}{315}\sqrt{2}-\frac{1216}{315}).
\end{multline*}

It is straightforward to check that $\rank R_\CYL(G,p)=19$, that
$(\omega,\lambda)\cdot R_\CYL(G,p)=0$  and that $\rank
\Omega_\CYL(\omega,\lambda)=15$.

\end{document}